\documentclass{amsart}
\usepackage[english]{babel}
\usepackage[latin1]{inputenc}
 \usepackage[all]{xy}

\usepackage{amsmath,amsfonts,amssymb,amsthm,amscd,array,stmaryrd,mathrsfs, mathdots}
\PassOptionsToPackage{option}{xcolor}
\usepackage{pstricks}

\theoremstyle{plain}
\newtheorem{fac}{Fact}
\newtheorem{lem}{Lemma}[section]
\newtheorem{thm}[lem]{Theorem}
\newtheorem{cor}[lem]{Corollary}
\newtheorem{prop}[lem]{Proposition}

\theoremstyle{definition}
\newtheorem{rem}[lem]{Remark}
\newtheorem{ex}[lem]{Example}

\newtheorem{defn}[lem]{Definition}

\let\ssection=\section
\renewcommand{\section}{\setcounter{equation}{0}\ssection}

\newcommand{\R}{\mathbb{R}}
\newcommand{\Z}{\mathbb{Z}}
\newcommand{\C}{\mathbb{C}}

\newcommand{\Q}{\mathbb{Q}}

\newcommand{\RP}{{\mathbb{RP}}}

\newcommand{\cM}{\mathcal{M}}

\newcommand{\id}{\textup{Id}}
\newcommand{\tr}{\textup{tr }}

\newcommand{\Id}{\mathrm{Id}}
\newcommand{\SL}{\mathrm{SL}}
\newcommand{\PSL}{\mathrm{PSL}}

\begin{document}

\title[Farey boat]{Farey boat:
Continued fractions and triangulations,\\
modular group and polygon dissections}

\author{Sophie Morier-Genoud, Valentin Ovsienko}

\address{Sophie Morier-Genoud,
Sorbonne Universit\'e, Universit\'e Paris Diderot, CNRS,
Institut de Math\'ematiques de Jussieu-Paris Rive Gauche,
 F-75005, Paris, France
}

\address{
Valentin Ovsienko,
Centre national de la recherche scientifique,
Laboratoire de Math\'ematiques 
U.F.R. Sciences Exactes et Naturelles 
Moulin de la Housse - BP 1039 
51687 Reims cedex 2,
France}
\email{sophie.morier-genoud@imj-prg.fr, valentin.ovsienko@univ-reims.fr}



\begin{abstract}
We reformulate several known results about continued fractions in combinatorial terms.
Among them the theorem of Conway and Coxeter
and that of Series, both relating continued fractions and triangulations.
More general polygon dissections appear when
extending these theorems for elements of the modular group $\PSL(2,\Z)$.
These polygon dissections are interpreted as 
walks in the Farey tessellation.
The combinatorial model of continued fractions can be further developed to
obtain a canonical presentation of elements of $\PSL(2,\Z)$.
\end{abstract}

\maketitle

\hfill{ \it \`A la m\'emoire de Christian Duval}

\thispagestyle{empty}

\section*{Introduction}\label{IntSec}

In this paper we formulate combinatorial interpretations of algebraic properties 
of continued fractions and of matrices in the modular group~$\PSL(2,\Z)$. 
The combinatorics is related to polygon dissections and walks in the Farey tessellation.

The starting point of the present paper is a theorem due to Conway and Coxeter~\cite{CoCo}.
This theorem uses triangulations of polygons to classify Coxeter's ``frieze patterns''.
Work of Coxeter on frieze patterns was motivated by
continued fractions; see~\cite{Cox}.
Our first goal is to reformulate Conway and Coxeter's theorem 
(and related notions) directly in terms of continued fractions, and to compare it
to some known results in the area.
In particular, we compare the Conway and Coxeter theorem with the theorem of Series~\cite{Ser}
that provides an embedding of continued fractions into the Farey tessellation.
This comparison offers a combinatorial relation between negative and regular continued fractions.

The second goal of the paper is to develop the combinatorics that arose
from the above comparison.
This leads to surprising results and notions, that appeared recently 
in the literature~\cite{CoOv,Ovs}.
Among them are relationship between continued fractions and Pfaffians of skew-symmetric matrices, and to
some particular polygon dissections.
We give a survey of this recent development.
Furthermore, along the same lines, we obtain several statements that appear to be new.
These are Theorems~\ref{FWalkThm}, \ref{PPCompleteThm}, \ref{PPCompleteThmBis} and~\ref{SLDecThm}.
We understand elements of~$\PSL(2,\Z)$ as generalized (finite) continued fractions,
triangulations are replaced by more general polygon dissections.

Let us outline possible applications and further developments of
the combinatorial approach discussed in this paper.
We believe that the relation between~$\PSL(2,\Z)$ and polygon dissections 
(see, in particular, Theorems~\ref{SecondMainThm} and~\ref{SLDecThm})
can be applied to other groups extending $\PSL(2,\Z)$. 
This relation connects the topic with several other
areas of algebra, geometry and combinatorics (such as cluster algebras, frieze patterns, etc.).
Application and combination of various methods known in these areas look promising.
One application is already explored in the second part of this work~\cite{MGOprep}, where
we suggest a notion of $q$-deformed continued fractions
and of $q$-deformed rational numbers;
this deformation preserves the combinatorial properties
discussed in the present paper.\\

The paper consists of six sections, each of them can be read independently.

In Section~\ref{TriangCFSec} we expose a combinatorial model for continued fractions.
We consider two classically known expansions of a rational number
$$
\frac{r}{s}
\quad=\quad
c_1 - \cfrac{1}{c_2 
          - \cfrac{1}{\ddots - \cfrac{1}{c_k} } } 
         \quad =\quad
a_1 + \cfrac{1}{a_2 
          + \cfrac{1}{\ddots +\cfrac{1}{a_{2m}} } } ,
$$
with $c_i\geq2$ and $a_i\geq1$.
The algebraic relationship between these two expansions due to Hirzebruch~\cite{Hir} is encoded 
in a triangulation of a polygon.
Although this section is introductory, it contains the main tools
used throughout the paper, such as Ptolemy rule and triangulations of polygons in the Farey graph.
The statements in this section are essentially reformulations of results that can be found in terms of frieze patterns in
Coxeter~\cite{Cox} and results in terms of hyperbolic geometry in Series~\cite{Ser}.
We call these statements ``Facts'' 
 and illustrate them on running examples.

In Section~\ref{MatHZBigSec} we focus on the matrices
\begin{equation}
\label{SLMatEqInt}
M(c_1,\ldots,c_k):=
\left(
\begin{array}{cc}
c_1&-1\\[4pt]
1&0
\end{array}
\right)
\left(
\begin{array}{cc}
c_{2}&-1\\[4pt]
1&0
\end{array}
\right)
\cdots
\left(
\begin{array}{cc}
c_k&-1\\[4pt]
1&0
\end{array}
\right)
\end{equation}
and
$$
M^+(a_1,\ldots,a_{2m}):=
\left(
\begin{array}{cc}
a_1&1\\[4pt]
1&0
\end{array}
\right)
\left(
\begin{array}{cc}
a_2&1\\[4pt]
1&0
\end{array}
\right)
\cdots
\left(
\begin{array}{cc}
a_{2m}&1\\[4pt]
1&0
\end{array}
\right)
$$
associated with the continued fractions. We establish elementary algebraic properties of these matrices and in particular their algebraic relationship.
In this section, the remarkable identity $M(c_1,\ldots,c_n)=-\Id$ appears using the combinatorial data introduced in Section~\ref{TriangCFSec}. 

In Section~\ref{CoCoIntSec} we describe combinatorially the complete set of positive integer $n$-tuples 
$(c_1,\ldots,c_n)$ that are solutions of the equation 
\begin{equation}
\label{SLCoCoEq}
M(c_1,\ldots,c_n)=
\pm\Id.
\end{equation}
The theorem of Conway and Coxeter~\cite{CoCo} provides a certain subset of solutions of 
$M(c_1,\ldots,c_n)=-\Id$ in terms of triangulations of $n$-gons. 
These solutions are obtained from the triangulations by counting  the number of triangles incident 
at each vertex of the $n$-gon.
All positive integer solutions of~\eqref{SLCoCoEq} are obtained
from a special class of dissections of $n$-gons called ``$3d$-dissections''~\cite{Ovs}.
Under weaker conditions on the coefficients~$c_i$, the continued
fraction disappears gradually, 
but the corresponding combinatorics lives on\footnote{
...similar to Cheshire cat's grin.}
and becomes more sophisticated.

Let us be a little bit more technical and briefly explain 
the way combinatorics appears in the context of
the modular group.
This relationship is central for the whole paper.
The standard choice of generators of~$\PSL(2,\Z)$ is
$$
R=\left(
\begin{array}{cc}
1&1\\[4pt]
0&1
\end{array}
\right),
\qquad
S=\left(
\begin{array}{cc}
0&-1\\[4pt]
1&0
\end{array}
\right),
$$
and all the relations in~$\PSL(2,\Z)$ are consequences of the following two relations:
$S^2=\Id$ and~$(RS)^3=\Id$,
implying the well-known isomorphism 
$\PSL(2,\Z)\simeq\left(\Z/2\Z\right)*\left(\Z/3\Z\right)$.
It is then not difficult to deduce that every $A\in\PSL(2,\Z)$ can be written
(non-uniquely) in the form 
\begin{equation}
\label{ItroDec}
A=R^{c_1}S\,R^{c_2}S\cdots{}R^{c_n}S,
\end{equation}
in such a way that all the coefficients $c_i$ are 
{\it positive integers}.
Note that~\eqref{ItroDec} coincides with~\eqref{SLMatEqInt},
i.e., $A=M(c_1,\ldots,c_n)$.
It is a rule in combinatorics that positive integers count some objects, and
Theorem~\ref{SecondMainThm} provides this interpretation
in the case where~$A$ is a {\it relation} in~$\PSL(2,\Z)$,
i.e., when~$A=\Id$.

In Section~\ref{FarSec} solutions of~(\ref{SLCoCoEq}) are embedded into the 
Farey graph. 
The embedding makes use of the sequence of rationals defined as
the convergents of the negative continued fraction corresponding to a positive solution.
Conway and Coxeter's solutions are then identified with ``Farey polygons''
(as proved in~\cite{SVS}).
More general solutions correspond to ``walks on Farey polygons''.

In Section~\ref{PPPSec} we connect the topic to the Ptolemy-Pl\"ucker relations
(and thus to cluster algebras; see~\cite{FZ,FZ1}).
The origin of these considerations goes back to
Euler who proved a series of identities for the ``continuants'', i.e., the polynomials describing continued
fractions in terms of the coefficients~$c_i$ (or~$a_i$).
Following~\cite{Ust}, we interpret Euler's identity in terms of the
Pfaffian of a $4\times4$ skew-symmetric matrix.
We also give
the ``Pfaffian formula'' (obtained in~\cite{CoOv}) for the trace of the matrix $M(c_1,\ldots,c_n)$.
Note that this appearance of Pfaffians is not a simple artifact,
it reflects a relationship between the subject and symplectic geometry; see~\cite{CoOv1}.
However, we do not describe this relationship in the present paper.

Section~\ref{DecSec} formulates some consequences of the developed combinatorics
for the modular group $\PSL(2,\Z)$.
Every element $A$ of~$\PSL(2,\Z)$
can be written in the form $A=M(c_1,\ldots,c_k)$ in infinitely many different ways.
We make such a presentation canonical by imposing the conditions~$c_i\geq1$ and~$k$ being the smallest possible, and deduce presentations of~$A$ in the standard generators of~$\PSL(2,\Z)$.
We prove that the canonical presentation $A=M(c_1,\ldots,c_k)$
is given by the expansion into the negative continued fraction of the quotient of greatest coefficients of~$A$.
Matrices $M(c_1,\ldots,c_k)$ with $c_i\geq2$ were used to parametrize conjugacy classes of~$\PSL(2,\Z)$; see~\cite{HZ,Zag,Kat},
the sequence $(c_1,\ldots,c_k)$ being the period of the continued fraction of a fixed point of~$A$, which is a quadratic irrational.
In our approach the quadratic irrational is replaced by a rational number.

\tableofcontents


\section{Continued fractions and triangulations}\label{TriangCFSec}

This section is a collection of basic properties of continued fractions
that we formulate in a combinatorial manner.

Let $r$ and $s$ be two coprime positive integers,
and assume that $r>s$.
The rational number~$\frac{r}{s}$ has unique expansions
\begin{equation}
\label{NegRegCFEq}
\frac{r}{s}
\quad=\quad
c_1 - \cfrac{1}{c_2 
          - \cfrac{1}{\ddots - \cfrac{1}{c_k} } } 
         \quad =\quad
a_1 + \cfrac{1}{a_2 
          + \cfrac{1}{\ddots +\cfrac{1}{a_{2m}} } } ,
\end{equation}
where $c_i\geq2$ and $a_i\geq1$, for all $i$.

The first expansion is usually called
a {\it negative}, or reversal continued fraction
the second is a (more common) {\it regular continued fraction}.
We will use the notation
$\llbracket{}c_1,\ldots,c_k\rrbracket{}$
and
$[a_1,\ldots,a_{2m}]$ 
for the above continued fractions, respectively.
Note that one can always assume the number of terms in the regular
continued fraction to be even, since
$[a_1,\ldots,a_{\ell}+1]=[a_1,\ldots,a_{\ell},1]$.

The explicit formula to obtain the coefficients~$(c_1,\ldots,c_k)$
in terms of the coefficients~$(a_1,\ldots,a_{2m})$, whenever~$\llbracket{}c_1,\ldots,c_k\rrbracket{}=[a_1,\ldots,a_{2m}]$, 
is as follows:
\begin{equation}
\label{HZRegEqt}
(c_1,\ldots,c_k)=
\big(a_1+1,\underbrace{2,\ldots,2}_{a_2-1},\,
a_3+2,\underbrace{2,\ldots,2}_{a_4-1},\ldots,
a_{2m-1}+2,\underbrace{2,\ldots,2}_{a_{2m}-1}\big).
\end{equation}
This expression can be found in~\cite[Eq. (19), p.241]{Hir} and~\cite[Eqs. (22), (23)]{HZ}, see also~\cite[p.93]{Never}.
We will give a combinatorial explanation of this formula.
In Section~\ref{MatHZSec} we will give a detailed proof
of a more general statement.

The goal of this introductory section is to explain that
both, regular and negative, continued fractions can be encoded by 
the same simple combinatorial picture.
We will be considering triangulated $n$-gons with exactly two exterior triangles.
All the statements of this section are
combinatorial reformulations of known results.

\subsection{Triangulations with two exterior triangles}\label{TriDvaSec}

Given a (convex) $n$-gon, we will be considering the classical notion of {\it triangulation}
which is a maximal dissection of the $n$-gon by diagonals
that never cross except for the endpoints.
A triangle in a triangulation is called {\it exterior} if
two of its sides are also sides (and not diagonals) of the $n$-gon.

In this section, we consider only those triangulations
that have exactly two exterior triangles.
In such a triangulation
the diagonal connecting the exterior vertices of the exterior triangles
has the property to cross every diagonal of the triangulation:
$$
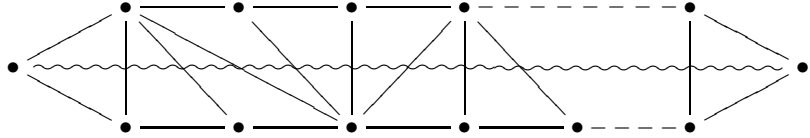
\xymatrix @!0 @R=0.8cm @C=1.5cm
{
&\bullet\ar@{-}[ld]\ar@{-}[dd]\ar@{-}[rdd]\ar@{-}[rrdd]\ar@{-}[r]
&\bullet\ar@{-}[r]
\ar@{-}[rdd]
&\bullet\ar@{-}[r]\ar@{-}[dd]
&\bullet\ar@{-}[ldd]\ar@{-}[dd]\ar@{-}[rdd]
\ar@{--}[rr]&
&\bullet\ar@{-}[rd]\ar@{-}[dd]&\\
\bullet\ar@{-}[rd]\ar@{~}[rrrrrrr]&&&&&&&\bullet\ar@{-}[ld]
\\
&\bullet\ar@{-}[r]
&\bullet\ar@{-}[r]
&\bullet\ar@{-}[r]&\bullet\ar@{-}[r]&\bullet\ar@{--}[r]&\bullet&
}
$$
Then, every triangle in the triangulation (except for the exterior ones)
can be situated with respect to this diagonal
in one of the two possible ways:
$$
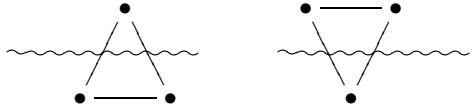
\xymatrix @!0 @R=0.6cm @C=0.6cm
{
&&&\bullet\ar@{-}[ldd]\ar@{-}[rdd]&&&&\bullet\ar@{-}[rr]&&\bullet\\
\ar@{~}[rrrrr]&&&&&&\ar@{~}[rrrrr]&&&&&\\
&&\bullet\ar@{-}[rr]&&\bullet&&&&\bullet\ar@{-}[luu]\ar@{-}[ruu]&
}
$$
that we refer to as  ``base-down'' or ``base-up''.
We assume the first exterior triangle to be situated base-down,
and the last one base-up.

We enumerate the vertices from $0$ to $n-1$ in a (clockwise) cyclic order:
$$
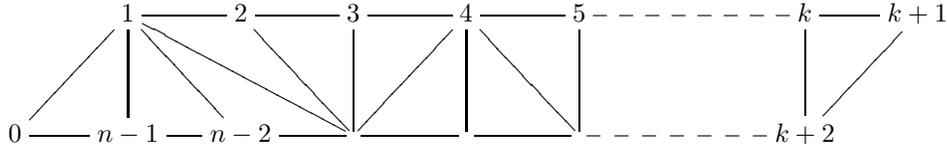
\xymatrix @!0 @R=0.8cm @C=1.5cm
{
&1\ar@{-}[ldd]\ar@{-}[dd]\ar@{-}[rdd]\ar@{-}[rrdd]\ar@{-}[r]&2\ar@{-}[r]
\ar@{-}[rdd]&3\ar@{-}[r]\ar@{-}[dd]&4\ar@{-}[r]\ar@{-}[ldd]\ar@{-}[dd]\ar@{-}[rdd]
&5\ar@{--}[rr]\ar@{-}[dd]&& k\ar@{-}[r]\ar@{-}[dd]& {k+1}\ar@{-}[ldd]\\
\\
0\ar@{-}[r]
&{n-1}\ar@{-}[r]&n-2\ar@{-}[r]&\ar@{-}[r]&\ar@{-}[r]&\ar@{--}[rr]&&{k+2}&
}
$$
so that the exterior vertices are $0$ and $k+1$.

\subsection{Combinatorial interpretation of continued fractions}\label{CombIntCFSec}

Given an $n$-gon and its triangulation
with two exterior triangles, we fix the following notation.
\begin{enumerate}
\item
The integers $(a_1,a_2,\ldots,a_{2m})$ count the number of equally positioned triangles,
i.e. the triangulation consists of the concatenation of $a_1$ triangles base down, followed by $a_2$ triangles base up and so on:
$$
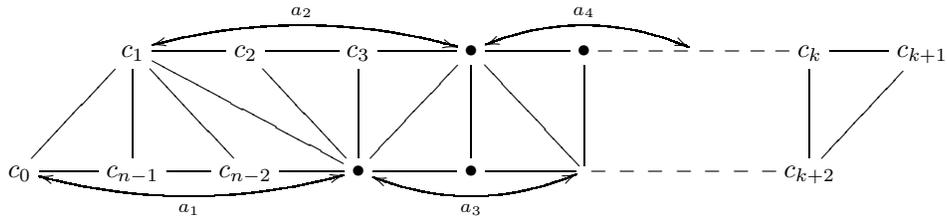
\xymatrix @!0 @R=0.8cm @C=1.5cm
{
&c_1\ar@{-}[ldd]\ar@{-}[dd]\ar@{-}[rdd]\ar@{-}[rrdd]\ar@{-}[r]\ar@/^0.8pc/@{<->}[rrr]^{a_2}
&c_2\ar@{-}[r]
\ar@{-}[rdd]&c_3\ar@{-}[r]\ar@{-}[dd]
&\bullet\ar@{-}[r]\ar@{-}[ldd]\ar@{-}[dd]\ar@{-}[rdd]\ar@{-}[r]\ar@/^0.8pc/@{<->}[rr]^{a_4}
&\bullet\ar@{--}[rr]\ar@{-}[dd]&& c_k\ar@{-}[r]\ar@{-}[dd]& c_{k+1}\ar@{-}[ldd]\\
\\
c_0\ar@{-}[r]\ar@/_0.8pc/@{<->}[rrr]_{a_1}
&c_{n-1}\ar@{-}[r]&c_{n-2}\ar@{-}[r]
&\bullet\ar@{-}[r]\ar@/_0.8pc/@{<->}[rr]_{a_3}
&\bullet\ar@{-}[r]&\ar@{--}[rr]&&c_{k+2}&
}
$$

\item
The integers $\left(c_1,c_2,\ldots,c_n=c_0\right)$ 
count the number of triangles at each vertex, 
i.e., the integer~$c_i$ is the number of triangles incident to the vertex $i$.
\end{enumerate}

Formula \eqref{HZRegEqt} is equivalent to the fact that these sequences define the same rational number.

\begin{fac}
\label{NRCFThm}
If $(a_1,\ldots,a_{2m})$ and $(c_1,\ldots,c_k)$
are the integers defined by (1) and (2), respectively,
then they are the coefficients of the expansions of the same rational number
as a regular and negative continued fraction, i.e.,
$$
[a_1,\ldots,a_{2m}]=\llbracket{}c_1,\ldots,c_k\rrbracket{}.
$$
\end{fac}

For a proof, see Section~\ref{MatHZSec}.

It is clear that each of the data $(a_1,\ldots,a_{2m})$ 
and $(c_1,\ldots,c_k)$ defines uniquely (the same) triangulation
of a polygon with two exterior triangles.
The number $n$ of vertices is related to the sequences via
$$
a_1+a_2+\cdots+a_{2m}=n-2,
\qquad
c_1+c_2+\cdots+c_k=n+k-3.
$$
Fact~\ref{NRCFThm} then implies the following.

\begin{cor}
\label{NRCFThmBis}
The set of rationals $\frac{r}{s}>1$
is in a one-to-one correspondence with triangulations of polygons
 with two exterior triangles.
\end{cor}

\begin{defn}
{\rm
Given a rational number~$\frac{r}{s}>1$,
we denote by~$\mathbb{T}_{r/s}$
the corresponding triangulation with two exterior triangles.
}
\end{defn}

\begin{ex}
\label{ExT75}
One has
$$
\frac{7}{5}=
[1,2,1,1]=
\llbracket{}2,2,3\rrbracket.
$$
The corresponding triangulation $\mathbb{T}_{7/5}$ is
$$
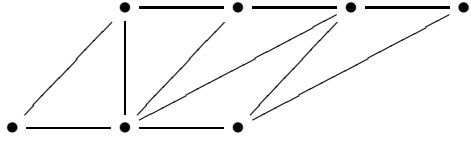
\xymatrix @!0 @R=0.8cm @C=1.5cm
{
&\bullet\ar@{-}[ldd]\ar@{-}[dd]\ar@{-}[r]
&\bullet\ar@{-}[ldd]\ar@{-}[r]
&\bullet\ar@{-}[lldd]\ar@{-}[ldd]\ar@{-}[r]
&\bullet\ar@{-}[lldd]\\
\\
\bullet\ar@{-}[r]
&\bullet\ar@{-}[r]
&\bullet
}
$$

\end{ex}

\subsection{The mirror formula}\label{MirFSec}

Consider the reversal of a regular continued fraction:
$[a_{2m},a_{2m-1},\ldots,a_1]$,
which is important in number theorey; see, e.g.,~\cite{AA}.

For every $\ell\geq0$, define the $\ell^{\rm th}$ {\it convergent} of the regular continued fraction~$[a_1,\ldots,a_{2m}]$ by
$$
\frac{r_\ell}{s_\ell}
\quad:=\quad
a_1 + \cfrac{1}{a_2 
          + \cfrac{1}{\ddots +\cfrac{1}{a_{\ell}} } } .
$$
The convergents of the negative continued fraction are defined in a similar way.

The following statement is known as the ``mirror formula'':
$$
\frac{r_{2m}}{r_{2m-1}}=[a_{2m},a_{2m-1},\ldots,a_1].
$$
In Section~\ref{MatConverSec}, we will prove this statement with the help 
of the matrix form of continued fractions.

The conversion into a negative continued fraction resorts to the
coefficients $c_i$ on the opposite vertices of~$\mathbb{T}_{r/s}$.

\begin{cor}
\label{CorRot}
One has
$$
[a_{2m},a_{2m-1},\ldots,a_1]=
\llbracket{}c_{k+2},c_{k+3},\ldots,c_{n-1}\rrbracket.
$$
\end{cor}

\begin{proof}
This formula follows from Fact~\ref{NRCFThm} 
when ``rotating'' the triangulation~$\mathbb{T}_{r/s}$.
\end{proof}

\begin{ex}
\label{ExT75Rev}
The reversal of the continued fraction from Example~\ref{ExT75} is as follows:
$$
\frac{7}{4}=
[1,1,2,1]=
\llbracket{}2,4\rrbracket.
$$

\end{ex}

\subsection{Farey sums and the labeling of vertices}\label{FSLabSec}
The rational~$\frac{r}{s}$ can be recovered from the triangulation~$\mathbb{T}_{r/s}$ by an additive rule.

Let us label the vertices of the $n$-gon by elements of the set~$\Q\cup\left\{\frac{1}{0}\right\}$.
We start from $\frac{0}{1}$ and $\frac{1}{0}$
at vertices $0$ and $1$, respectively.
We then extend this labeling to the whole $n$-gon 
by the following ``Farey summation formula''.
Whenever two vertices of the same triangle have been assigned the rationals
$\frac{r'}{s'}$ and $\frac{r''}{s''}$, then the third vertex receives the label
$$
\frac{r'}{s'}\oplus\frac{r''}{s''}
:=\frac{r'+r''}{s'+s''}.
$$
This process is illustrated by the following example.
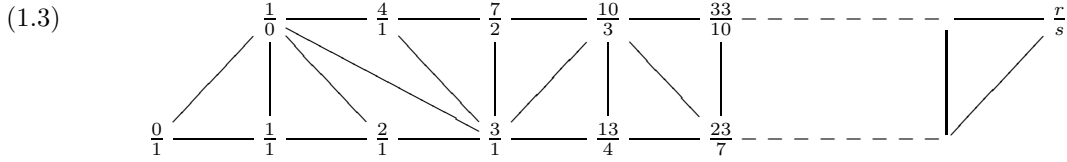
\begin{equation}
\label{FFracEq}
\xymatrix @!0 @R=0.8cm @C=1.5cm
{
&\frac{1}{0}\ar@{-}[ldd]\ar@{-}[dd]\ar@{-}[rdd]\ar@{-}[rrdd]\ar@{-}[r]
&\frac{4}{1}\ar@{-}[r]
\ar@{-}[rdd]&\frac{7}{2}\ar@{-}[r]\ar@{-}[dd]&\frac{10}{3}\ar@{-}[r]\ar@{-}[ldd]\ar@{-}[dd]\ar@{-}[rdd]
&\frac{33}{10}\ar@{--}[rr]\ar@{-}[dd]&
& \ar@{-}[r]\ar@{-}[dd]
& \frac{r}{s}\ar@{-}[ldd]\\
\\
\frac{0}{1}\ar@{-}[r]
&\frac{1}{1}\ar@{-}[r]&\frac{2}{1}\ar@{-}[r]&\frac{3}{1}\ar@{-}[r]
&\frac{13}{4}\ar@{-}[r]
&\frac{23}{7}\ar@{--}[rr]&
&&
}
\end{equation}

The following statement is easily proved by induction.
It can be viewed as a reformulation of the result of Series~\cite{Ser};
for more details, see Section~\ref{TrsFarSec}.

\begin{fac}
\label{CaroThm}
Labeling the vertices of the triangulation $\mathbb{T}_{r/s}$
according to the above rule, the vertex $k+1$ receives the label~$\frac{r}{s}$.
\end{fac}

\begin{rem}
More generally,
 all the rationals labeling the vertices $2,3,\ldots,k,{k+1}$
are the consecutive convergents of the negative continued fraction
$\llbracket{}c_1,\ldots,c_k\rrbracket{}$
representing $\frac{r}{s}$.
\end{rem}

\subsection{Recovering $r$ and $s$ with the Ptolemy-Pl\"ucker rule}\label{CFPtoRuSec}

In Euclidean geometry,
the Ptolemy relation is the formula relating the lengths of
the diagonals and sides of an inscribed quadrilateral.
It reads
$$
x_{1,3}x_{2,4}=
x_{1,2}x_{3,4}+x_{2,3}x_{4,1},
$$
where $x_{i,j}$ is the Euclidean length between the vertices $i$ and $j$.
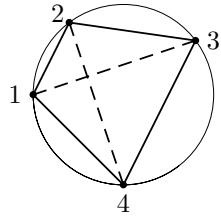
\begin{center}
\psscalebox{1.0 1.0} 
{\psset{unit=0.6cm}
\begin{pspicture}(0,-2.26)(4.535,2.26)
\psarc[linecolor=black, linewidth=0.02, dimen=outer](2.46,0.25){2.0}{-180.0}{308.0}
\psdots[linecolor=black, dotsize=0.16](1.26,1.85)
\psdots[linecolor=black, dotsize=0.16](0.46,0.25)
\psdots[linecolor=black, dotsize=0.16](2.46,-1.75)
\psdots[linecolor=black, dotsize=0.16](4.06,1.45)
\psline[linecolor=black, linewidth=0.04](1.26,1.85)(0.46,0.25)(2.46,-1.75)(4.06,1.45)(1.26,1.85)
\psline[linecolor=black, linewidth=0.04, linestyle=dashed, dash=0.17638889cm 0.10583334cm](1.26,1.85)(2.46,-1.75)
\psline[linecolor=black, linewidth=0.04, linestyle=dashed, dash=0.17638889cm 0.10583334cm](0.46,0.25)(4.06,1.45)
\rput(0.06,0.25){1}
\rput[tl](0.86,2.25){2}
\rput(4.46,1.45){3}
\rput(2.46,-2.15){4}
\end{pspicture}
}

\end{center}
In algebraic geometry and combinatorics,
the Ptolemy relations appear as the relations between the 
Pl\"ucker coordinates of the Grassmannian $Gr_{2,n}$, so that they are often called Ptolemy-Pl\"ucker relations.
We will use this name in the sequel.
They became an important and general rule in the theory of
cluster algebras~\cite{FZ,FZ1}.
The ``Ptolemy-Pl\"ucker rule'' provides a way to calculate new variables from the old ones.

Let us explain how the Ptolemy rule
allows one to calculate the numerator~$r$ and the denominator~$s$ 
of the continued fraction~(\ref{NegRegCFEq})
from the corresponding triangulation~$\mathbb{T}_{r/s}$.

Given a triangulated $n$-gon with exactly two exterior triangles,
we will assign a value~$x_{i,j}$ to all its edges~$(i,j)$ with $i\leq{}j$,
so that the system of equations
\begin{equation}
\label{PPCompleteEq}
\left\{
\begin{array}{rcl}
x_{i,j}x_{k,\ell}&=&x_{i,k}x_{j,\ell}+x_{i,\ell}x_{k,j},
\qquad
i\leq{}k\leq{}j\leq{}\ell,\\[6pt]
x_{i,i}&=&0,
\end{array}
\right.
\end{equation}
is satisfied.
The system~(\ref{PPCompleteEq}) will be called the {\it Ptolemy-Pl\"ucker relations}.

\begin{fac}
\label{KouchThm}
(i)
The labels $x_{i,j}$ satisfying~(\ref{PPCompleteEq})
are uniquely determined by the values $x_{i,j}$ 
of the sides and diagonals of the triangulation.

(ii)
Assume that~$x_{i,j}=1$
whenever~$(i,j)$ is a side or a 
diagonal of the triangulation
$$
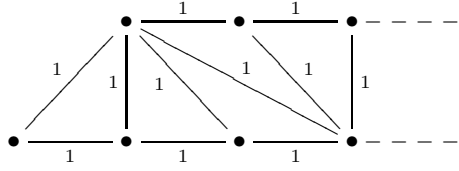
\xymatrix @!0 @R=0.8cm @C=1.5cm
{
&\bullet\ar@{-}[ldd]_{1}\ar@{-}[dd]_1\ar@{-}[rdd]_<<<<<<<1\ar@{-}[rrdd]^1\ar@{-}[r]^1
&\bullet\ar@{-}[r]^1
\ar@{-}[rdd]^1&\bullet\ar@{--}[r]\ar@{-}[dd]^1&\\
\\
\bullet\ar@{-}[r]_1
&\bullet\ar@{-}[r]_1&\bullet\ar@{-}[r]_1&\bullet\ar@{--}[r]
&
}
$$
Then all the labels~$x_{i,j}$ are positive integers.

(iii)
In the triangulation~$\mathbb{T}_{r/s}$,
the assumption from Part~(ii) implies the labeling 
$$
\left\{
\begin{array}{rcl}
x_{0,k+1} &= &r\\[4pt]
x_{1,k+1} &= &s.
\end{array}
\right.
$$
\end{fac}

Parts~(i) and~(ii) are widely known in the theory of cluster algebra;
see~\cite[Section~2.1.1]{GSV}.
We do not dwell on the proof here.

Part~(iii) was already known to Coxeter~\cite[Eq.(5.6)]{Cox}
who proved (in a different context) the following more general statement.

\begin{fac}
\label{KouchGenThm}
Under the assumption that~$x_{i,j}=1$
whenever~$(i,j)$ is a side or a 
diagonal of the triangulation,
the integers $x_{i,j}$ are calculated as $2\times2$ determinants:
\begin{equation}
\label{RemPtKouchEq}
x_{i,j}=
\det\left(
\begin{array}{cc}
r_i&r_j\\[4pt]
s_i&s_j
\end{array}
\right)=r_is_j-s_ir_j,
\end{equation}
where~$\frac{r_i}{s_i}$ and~$\frac{r_j}{s_j}$ are the rationals labeling the vertices~$i$ and~$j$,
as in~(\ref{FFracEq}).
\end{fac}

We will prove yet a more general result in Section~\ref{PtoSec}
(see Theorem~\ref{PPCompleteThm}).

We illustrate the statement~(iii) by the following diagram.
$$
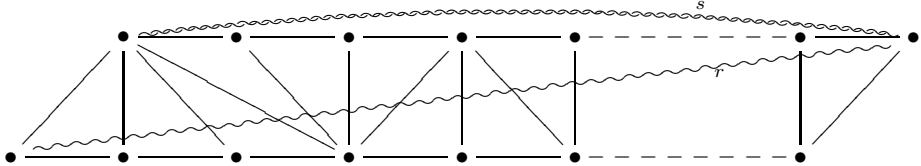
\xymatrix @!0 @R=0.8cm @C=1.5cm
{
&\bullet\ar@{-}[ldd]\ar@{-}[dd]\ar@{-}[rdd]\ar@{-}[rrdd]\ar@{-}[r]\ar@/^0.8pc/@{~}[rrrrrrr]^>>>>>>>>>>>>>s
&\bullet\ar@{-}[r]
\ar@{-}[rdd]&\bullet\ar@{-}[r]\ar@{-}[dd]&\bullet\ar@{-}[r]\ar@{-}[ldd]\ar@{-}[dd]\ar@{-}[rdd]
&\bullet\ar@{--}[rr]\ar@{-}[dd]&
& \bullet\ar@{-}[r]\ar@{-}[dd]
& \bullet\ar@{-}[ldd]\\
\\
\bullet\ar@{-}[r]\ar@{~}[uurrrrrrrr]_>>>>>>>>>>>>>>>>>>>>>>>>r
&\bullet\ar@{-}[r]&\bullet\ar@{-}[r]&\bullet\ar@{-}[r]
&\bullet\ar@{-}[r]
&\bullet\ar@{--}[rr]&
&\bullet&
}
$$

\subsection{Triangulations~$\mathbb{T}_{r/s}$ inside the Farey graph}\label{TrsFarSec}

The triangulation \eqref{FFracEq} can be naturally embedded in the Farey tessellation.
In this section we explain how to extract 
the triangulation~$\mathbb{T}_{r/s}$ from the Farey tessellation.
This construction is due to  C. Series~\cite{Ser},
and it allows one to deduce Fact~\ref{CaroThm} from her result.

\begin{defn}
{\rm
a)
The set of all rational numbers $\Q$, completed by $\infty$
represented by~$\frac{1}{0}$, form a graph called the {\it Farey graph}.
Two rationals written as irreducible fractions,~$\frac{r'}{s'}$ and~$\frac{r''}{s''}$, 
are connected by an edge if and only if $r's''-r''s'=\pm1$.

b)
Including $\Q\cup\{\infty\}$ into the border of the hyperbolic half-plane~$H$,
the edges are often represented as geodesics of~$H$
(which is a half-circle) and the Farey graph splits $H$ into an infinite set of triangles
called the {\it Farey tessellation}.
}
\end{defn}

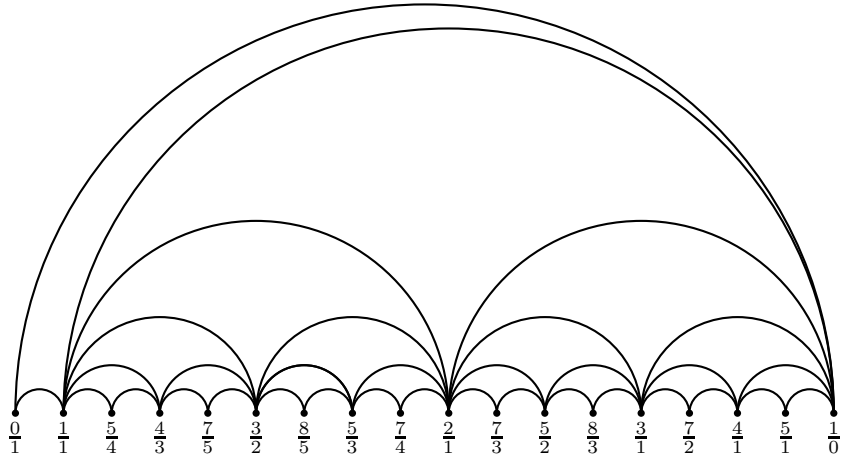
\begin{figure}[htbp]
\begin{center}

\psscalebox{1.0 1.0} 
{\psset{unit=0.8cm}
\begin{pspicture}(0,-3.7125)(13.757692,3.7125)
\psdots[linecolor=black, dotsize=0.12](3.2788463,-3.1075)
\psdots[linecolor=black, dotsize=0.12](2.478846,-3.1075)
\psdots[linecolor=black, dotsize=0.12](1.6788461,-3.1075)
\psdots[linecolor=black, dotsize=0.12](0.87884617,-3.1075)
\psdots[linecolor=black, dotsize=0.12](0.07884616,-3.1075)
\psdots[linecolor=black, dotsize=0.12](4.078846,-3.1075)
\psdots[linecolor=black, dotsize=0.12](4.878846,-3.1075)
\psdots[linecolor=black, dotsize=0.12](5.6788464,-3.1075)
\psdots[linecolor=black, dotsize=0.12](6.478846,-3.1075)
\psdots[linecolor=black, dotsize=0.12](7.2788463,-3.1075)
\psdots[linecolor=black, dotsize=0.12](8.078846,-3.1075)
\psdots[linecolor=black, dotsize=0.12](8.878846,-3.1075)
\psdots[linecolor=black, dotsize=0.12](9.678846,-3.1075)
\psdots[linecolor=black, dotsize=0.12](10.478847,-3.1075)
\psdots[linecolor=black, dotsize=0.12](11.278846,-3.1075)
\psdots[linecolor=black, dotsize=0.12](12.078846,-3.1075)
\psdots[linecolor=black, dotsize=0.12](12.878846,-3.1075)
\psarc[linecolor=black, dotsize=0.12, dimen=outer](6.878846,-3.1075){6.8}{0.0}{180.0}
\psdots[linecolor=black, dotsize=0.12](13.678846,-3.1075)
\psarc[linecolor=black, dotsize=0.12, dimen=outer](6.878846,-3.1075){0.4}{0.0}{180.0}
\psarc[linecolor=black, dotsize=0.12, dimen=outer](7.6788464,-3.1075){0.4}{0.0}{180.0}
\psarc[linecolor=black, dotsize=0.12, dimen=outer](8.478847,-3.1075){0.4}{0.0}{180.0}
\psarc[linecolor=black, dotsize=0.12, dimen=outer](9.278846,-3.1075){0.4}{0.0}{180.0}
\psarc[linecolor=black, dotsize=0.12, dimen=outer](10.078846,-3.1075){0.4}{0.0}{180.0}
\psarc[linecolor=black, dotsize=0.12, dimen=outer](10.878846,-3.1075){0.4}{0.0}{180.0}
\psarc[linecolor=black, dotsize=0.12, dimen=outer](11.678846,-3.1075){0.4}{0.0}{180.0}
\psarc[linecolor=black, dotsize=0.12, dimen=outer](12.478847,-3.1075){0.4}{0.0}{180.0}
\psarc[linecolor=black, dotsize=0.12, dimen=outer](13.278846,-3.1075){0.4}{0.0}{180.0}
\psarc[linecolor=black, dotsize=0.12, dimen=outer](0.47884616,-3.1075){0.4}{0.0}{180.0}
\psarc[linecolor=black, dotsize=0.12, dimen=outer](1.2788461,-3.1075){0.4}{0.0}{180.0}
\psarc[linecolor=black, dotsize=0.12, dimen=outer](2.0788462,-3.1075){0.4}{0.0}{180.0}
\psarc[linecolor=black, dotsize=0.12, dimen=outer](2.8788462,-3.1075){0.4}{0.0}{180.0}
\psarc[linecolor=black, dotsize=0.12, dimen=outer](3.6788461,-3.1075){0.4}{0.0}{180.0}
\psarc[linecolor=black, dotsize=0.12, dimen=outer](4.478846,-3.1075){0.4}{0.0}{180.0}
\psarc[linecolor=black, dotsize=0.12, dimen=outer](5.2788463,-3.1075){0.4}{0.0}{180.0}
\psarc[linecolor=black, dotsize=0.12, dimen=outer](6.078846,-3.1075){0.4}{0.0}{180.0}
\psarc[linecolor=black, dotsize=0.12, dimen=outer](1.6788461,-3.1075){0.8}{0.0}{180.0}
\psarc[linecolor=black, dotsize=0.12, dimen=outer](3.2788463,-3.1075){0.8}{0.0}{180.0}
\psarc[linecolor=black, dotsize=0.12, dimen=outer](4.878846,-3.1075){0.8}{0.0}{180.0}
\psarc[linecolor=black, dotsize=0.12, dimen=outer](6.478846,-3.1075){0.8}{0.0}{180.0}
\psarc[linecolor=black, dotsize=0.12, dimen=outer](8.078846,-3.1075){0.8}{0.0}{180.0}
\psarc[linecolor=black, dotsize=0.12, dimen=outer](9.678846,-3.1075){0.8}{0.0}{180.0}
\psarc[linecolor=black, dotsize=0.12, dimen=outer](11.278846,-3.1075){0.8}{0.0}{180.0}
\psarc[linecolor=black, dotsize=0.12, dimen=outer](12.878846,-3.1075){0.8}{0.0}{180.0}
\psarc[linecolor=black, dotsize=0.12, dimen=outer](2.478846,-3.1075){1.6}{0.0}{180.0}
\psarc[linecolor=black, dotsize=0.12, dimen=outer](4.878846,-3.1075){0.8}{0.0}{180.0}
\psarc[linecolor=black, dotsize=0.12, dimen=outer](5.6788464,-3.1075){1.6}{0.0}{180.0}
\psarc[linecolor=black, dotsize=0.12, dimen=outer](8.878846,-3.1075){1.6}{0.0}{180.0}
\psarc[linecolor=black, dotsize=0.12, dimen=outer](12.078846,-3.1075){1.6}{0.0}{180.0}
\psarc[linecolor=black, dotsize=0.12, dimen=outer](4.078846,-3.1075){3.2}{0.0}{180.0}
\psarc[linecolor=black, dotsize=0.12, dimen=outer](10.478847,-3.1075){3.2}{0.0}{180.0}
\psarc[linecolor=black, dotsize=0.12, dimen=outer](7.2788463,-3.1075){6.4}{0.0}{180.0}
\rput(0.07884616,-3.5075){$\frac01$}
\rput(0.87884617,-3.5075){$\frac11$}
\rput(1.6788461,-3.5075){$\frac54$}
\rput(2.478846,-3.5075){$\frac43$}
\rput(3.2788463,-3.5075){$\frac75$}
\rput(4.078846,-3.5075){$\frac32$}
\rput(4.878846,-3.5075){$\frac85$}
\rput(5.6788464,-3.5075){$\frac53$}
\rput(6.478846,-3.5075){$\frac74$}
\rput(7.2788463,-3.5075){$\frac21$}
\rput(8.078846,-3.5075){$\frac73$}
\rput(8.878846,-3.5075){$\frac52$}
\rput(9.678846,-3.5075){$\frac83$}
\rput(10.478847,-3.5075){$\frac31$}
\rput(11.278846,-3.5075){$\frac72$}
\rput(12.078846,-3.5075){$\frac41$}
\rput(12.878846,-3.5075){$\frac51$}
\rput(13.678846,-3.5075){$\frac10$}
\end{pspicture}
}
\caption{A fragment of the Farey graph}
\label{default}
\end{center}
\end{figure}

Basic properties of the Farey graph and Farey tessellation can be found in~\cite{HW}, 
we will need the following.

a) The edges of the Farey tessellation never cross, except at the endpoints.

b) Every triangle in the Farey graph is of the form
$\left\{\frac{r'}{s'},\frac{r'+r''}{s'+s''},\frac{r''}{s''}\right\}$,
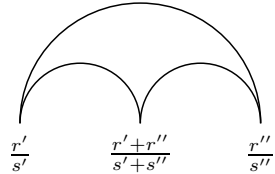
\begin{center}
\psscalebox{1.0 1.0} 
{
\begin{pspicture}(0,-1.315)(3.385,1.315)
\definecolor{colour0}{rgb}{1.0,0.0,0.2}
\psarc[linecolor=black, linewidth=0.02, dimen=outer](0.87,-0.685){0.8}{0.0}{180.0}
\psarc[linecolor=black, linewidth=0.02, dimen=outer](2.47,-0.685){0.8}{0.0}{180.0}
\psarc[linecolor=black, linewidth=0.02, dimen=outer](1.67,-0.685){1.6}{0.0}{180.0}
\rput(0.07,-1.085){$\frac{r'}{s'}$}
\rput(1.67,-1.085){$\frac{r'+r''}{s'+s''}$}
\rput(3.27,-1.085){$\frac{r''}{s''}$}
\end{pspicture}
}
\end{center}
We focus on the part of the graph consisting of rational numbers greater than 1. 

The construction of~\cite{Ser} is as follows.
Fix a rational number~$\frac{r}{s}$
and draw a vertical line $(L)\subset{}H$ through~$\frac{r}{s}$.
Collect all the triangles of the Farey tessellation crossed 
in their interior by this line. 
This leads to the triangulation $\mathbb{T}_{r/s}$.

The property  of the triangles
to be situated  ``base down'' and ``base up'' now read as:
 ''base at the left of $(L)$'' 
and ''base at the right of $(L)$''. 
The two exterior vertices are $\frac01$ and $\frac{r}{s}$.
The vertices are enumerated from 1 to $n$ from $\frac10$ to $\frac01$
in the decreasing order.
The vertex $\frac{r}{s}$ is the vertex number $k+1$.

\begin{ex}
Choosing $\frac{r}{s}=\frac{7}{5}$, we have the following picture:
\medskip
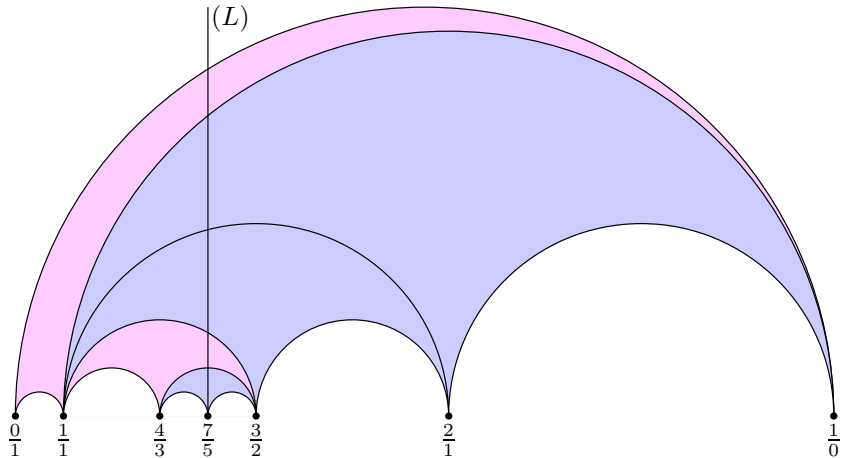
\begin{center}
\psscalebox{1.0 1.0} 
{\psset{unit=0.8cm}
\begin{pspicture}(0,-3.7125)(13.757692,3.7125)
\definecolor{colour0}{rgb}{1.0,0.8,1.0}
\definecolor{colour1}{rgb}{0.8,0.8,1.0}
\psarc[linecolor=black, linewidth=0.02, fillstyle=solid,fillcolor=colour0, dimen=outer](6.878846,-3.1075){6.8}{0.0}{180.0}
\psarc[linecolor=black, linewidth=0.02, fillstyle=solid, dimen=outer](0.47884616,-3.1075){0.4}{0.0}{180.0}
\psarc[linecolor=black, linewidth=0.02, fillstyle=solid,fillcolor=colour1, dimen=outer](7.2788463,-3.1075){6.4}{0.0}{180.0}
\rput(0.07884616,-3.5075){$\frac01$}
\rput(0.87884617,-3.5075){$\frac11$}
\rput(2.478846,-3.5075){$\frac43$}
\rput(3.2788463,-3.5075){$\frac75$}
\rput(4.078846,-3.5075){$\frac32$}
\rput(7.2788463,-3.5075){$\frac21$}
\rput(13.678846,-3.5075){$\frac10$}
\rput(3.65,3.5075){$(L)$}
\psarc[linecolor=black, linewidth=0.02, dimen=outer](4.078846,-3.1075){3.2}{0.0}{180.0}
\psarc[linecolor=black, linewidth=0.02, fillstyle=solid,fillcolor=colour0, dimen=outer](2.478846,-3.1075){1.6}{0.0}{180.0}
\psarc[linecolor=black, linewidth=0.02, fillstyle=solid, dimen=outer](1.6788461,-3.1075){0.8}{0.0}{180.0}
\psarc[linecolor=black, linewidth=0.02, fillstyle=solid,fillcolor=colour1, dimen=outer](3.2788463,-3.1075){0.8}{0.0}{180.0}
\psarc[linecolor=black, linewidth=0.02, fillstyle=solid, dimen=outer](10.478847,-3.1075){3.2}{0.0}{180.0}
\psarc[linecolor=black, linewidth=0.02, fillstyle=solid, dimen=outer](5.6788464,-3.1075){1.6}{0.0}{180.0}
\psarc[linecolor=black, linewidth=0.02, fillstyle=solid, dimen=outer](2.8788462,-3.1075){0.4}{0.0}{180.0}
\psarc[linecolor=black, linewidth=0.02, fillstyle=solid, dimen=outer](3.6788461,-3.1075){0.4}{0.0}{180.0}
\psdots[linecolor=black, dotsize=0.12](2.478846,-3.1075)
\psdots[linecolor=black, dotsize=0.12](4.078846,-3.1075)
\psdots[linecolor=black, dotsize=0.12](7.2788463,-3.1075)
\psdots[linecolor=black, dotsize=0.12](0.87884617,-3.1075)
\psdots[linecolor=black, dotsize=0.12](13.678846,-3.1075)
\psdots[linecolor=black, dotsize=0.12](0.07884616,-3.1075)
\psdots[linecolor=black, dotsize=0.12003673](3.2788463,-3.1075)
\psline[linecolor=black, linewidth=0.02, dash=0.17638889cm 0.10583334cm](3.2788463,-3.1075)(3.2788463,3.6925)
\end{pspicture}
}
\end{center}
where we have colored in pink the triangles at the left of $(L)$ and in blue those at the right of $(L)$. Note that the lowest triangle can be viewed either at the left or at the right of $(L)$. 
This is precisely the triangulation $\mathbb{T}_{7/5}$ (cf. Example~\ref{ExT75}) viewed inside the Farey tessellation. 
\end{ex}

\section{Matrices of negative and regular continued fractions}\label{MatHZBigSec}

It is convenient to use $2\times2$ matrices to represent continued fractions.
One reason is that the corresponding matrices belong to the group~$\SL(2,\Z)$
and allow the operations, such as multiplication, inverse, transposition; see~\cite{Poo}.
Another reason which is particularly important for us is that matrices
are more ``perennial'' than continued fractions.
They continue to exist when continued fractions are not well-defined 
(because of potential zeros in the denominators) and enjoy similar properties.

In this section, however, we still assume that the continued fractions are well-defined.
Consider, as in Section~\ref{TriangCFSec}, a rational number expanded
into continued fractions:
$$
\frac{r}{s}=
\llbracket{}c_1,\ldots,c_k\rrbracket=
[a_1,\ldots,a_{2m}].
$$
The information about these expansions is contained in the matrices
\begin{equation}
\label{SLEq}
M(c_1,\ldots,c_n):=
\left(
\begin{array}{cc}
c_1&-1\\[4pt]
1&0
\end{array}
\right)
\left(
\begin{array}{cc}
c_{2}&-1\\[4pt]
1&0
\end{array}
\right)
\cdots
\left(
\begin{array}{cc}
c_n&-1\\[4pt]
1&0
\end{array}
\right)
\end{equation}
and
\begin{equation}
\label{RegMatEq}
M^+(a_1,\ldots,a_{2m}):=
\left(
\begin{array}{cc}
a_1&1\\[4pt]
1&0
\end{array}
\right)
\left(
\begin{array}{cc}
a_2&1\\[4pt]
1&0
\end{array}
\right)
\cdots
\left(
\begin{array}{cc}
a_{2m}&1\\[4pt]
1&0
\end{array}
\right).
\end{equation}
Both matrices are elements of~$\SL(2,\Z)$.

The goal of this introductory section is to compare these two matrices and rewrite one from another.
This, in particular, implies formula~\eqref{HZRegEqt}.
The end of the section contains motivations for the sequel.

\subsection{The matrices of continued fractions}\label{MatConverSec}
The matrices~\eqref{SLEq} and~\eqref{RegMatEq} are known as the matrices of
continued fractions, because one has the following statement whose proof is elementary.
\begin{prop}
\label{MatConv}
One has
$$
M(c_1,\ldots,c_k)=
\left(
\begin{array}{cc}
r&-r'\\[4pt]
s&-s'
\end{array}
\right),
\qquad
M^+(a_1,\ldots,a_{2m})=
\left(
\begin{array}{cc}
r&r''\\[4pt]
s&s''
\end{array}
\right),
$$
where  $\frac{r}{s}=
[a_1,\ldots,a_{2m}]=
\llbracket{}c_1,\ldots,c_k\rrbracket$,
and where $\frac{r'}{s'}=
\llbracket{}c_1,\ldots,c_{k-1}\rrbracket$,
and $\frac{r''}{s''}=[a_1,\ldots,a_{2m-1}]$.
\end{prop}

Therefore, the matrices $M^+(a_1,\ldots,a_{2m})$ and $M(c_1,\ldots,c_k)$
have the same first column, but they are different.
There exists a simple relationship between these matrices.

\begin{prop}
\label{NegRegTheorem}
One has:
\begin{equation}
\label{NegRegMatTheEq}
M^+(a_1,\ldots,a_{2m})=
M(c_1,\ldots,c_k)\,R,
\end{equation}
where
$$
R=\left(
\begin{array}{cc}
1&1\\[4pt]
0&1
\end{array}
\right).
$$
\end{prop}

\begin{proof}
Formula~\eqref{NegRegMatTheEq} can be easily obtained using the results of the previous section.
Indeed, in the triangulation~$\mathbb{T}_{r/s}$ labeled as in ~\eqref{FFracEq},
we see that $\frac{r'}{s'},\frac{r}{s},\frac{r''}{s''}$ label the vertices
$k,k+1,k+2$, respectively.
This implies $\frac{r}{s}=\frac{r'}{s'}+\frac{r''}{s''}$ and hence~\eqref{NegRegMatTheEq}.
\end{proof}

Alternatively and independently, the relation between the matrices $M^+(a_1,\ldots,a_{2m})$ and $M(c_1,\ldots,c_k)$ 
can be established by elementary matrix computations. 
This will be done in the next sections.

\begin{ex}
\label{ExT75Bis}
Choosing, as in Example~\ref{ExT75},
the rational $\frac{r}{s}=\frac{7}{5}$,
one obtains
$$
M^+(1,2,1,1)=
\left(
\begin{array}{cc}
7&4\\[4pt]
5&3
\end{array}
\right),
\qquad
M(2,2,3)=
\left(
\begin{array}{cc}
7&-3\\[4pt]
5&-2
\end{array}
\right).
$$
Note that these matrices have different traces and therefore cannot be conjugacy equivalent.
\end{ex}

\subsection{Matrices $M^+(a_1,\ldots,a_{2m})$ and $M(c_1,\ldots,c_k)$ in terms of the generators}\label{MaGSec}
It will be useful to have the expressions of $M^+(a_1,\ldots,a_{2m})$ and $M(c_1,\ldots,c_k)$
in terms of the generators of~$\SL(2,\Z)$.
The following formulas are standard and can be found in many sources.

\begin{prop}
\label{MaTCompProp}
The matrices $M^+(a_1,\ldots,a_{2m})$ and $M(c_1,\ldots,c_k)$
have the following decompositions
\begin{eqnarray}
\label{RegMatGenEq}
M^+(a_1,\ldots,a_{2m})&=&
R^{a_1}L^{a_{2}}
R^{a_{3}}L^{a_{4}}\cdots{}
R^{a_{2m-1}}L^{a_{2m}},
\\[4pt]
\label{NegMatGenEq}
M(c_1,\ldots,c_k)&=&R^{c_1}S\,R^{c_{2}}S\cdots{}R^{c_k}S,
\end{eqnarray}
where
\begin{equation}
\label{RLSEq}
R=\left(
\begin{array}{cc}
1&1\\[4pt]
0&1
\end{array}
\right),
\qquad
L=\left(
\begin{array}{cc}
1&0\\[4pt]
1&1
\end{array}
\right),
\qquad
S=\left(
\begin{array}{cc}
0&-1\\[4pt]
1&0
\end{array}
\right).
\end{equation}
\end{prop}

For the sake of completeness, we give here an elementary proof.

\begin{proof}
Formula \eqref{RegMatGenEq} is obtained from the elementary computation
$$
\left(
\begin{array}{cc}
a_i&1\\[4pt]
1&0
\end{array}
\right)
\left(
\begin{array}{cc}
a_{i+1}&1\\[4pt]
1&0
\end{array}
\right)
=
\left(
\begin{array}{cc}
a_ia_{i+1}+1&a_i\\[4pt]
a_{i+1}&1
\end{array}
\right)=
\left(
\begin{array}{cc}
1&a_i\\[4pt]
0&1
\end{array}
\right)
\left(
\begin{array}{cc}
1&0\\[4pt]
a_{i+1}&1
\end{array}
\right)
=R^{a_i}L^{a_{i+1}}.
$$
Formula~(\ref{NegMatGenEq}) is obviously obtained from
$$
\left(
\begin{array}{cc}
c_i&-1\\[4pt]
1&0
\end{array}
\right)=R^{c_{i}}S.
$$
\end{proof}

\subsection{Converting the matrices}\label{MatHZSec}

The matrix $M^+(a_1,\ldots,a_{2m})$ with $a_i\geq1$
can be rewritten in the form~(\ref{SLEq}).

\begin{prop}
\label{NegRegProp}
One has:
\begin{equation}
\label{NegRegMat}
M^+(a_1,\ldots,a_{2m})=
-M\big(a_1+1,\underbrace{2,\ldots,2}_{a_2-1},\,
a_3+2,\underbrace{2,\ldots,2}_{a_4-1},\ldots,
a_{2m-1}+2,\underbrace{2,\ldots,2}_{a_{2m}},1,1\big).
\end{equation}
\end{prop}

Let us stress that~\eqref{NegRegMat} is equivalent to~\eqref{NegRegMatTheEq}
under the assumption that we already know formula~\eqref{HZRegEqt}.
However, our strategy is different, we use~\eqref{NegRegMat} to prove~\eqref{HZRegEqt}.
 
We will need the following lemma.

\begin{lem}
\label{TransMat}
One has
$R^a=-M(a+1,1,1)$ and $L^a=-M(1,\underbrace{2,\ldots,2}_{a},1,1).$
\end{lem}

\begin{proof}
With a direct computation one easily obtains $M(a+1,1,1)=-R^{a}$.
For the second formula we use the following preliminary result that is easily obtained
by induction
$$
\left(
\begin{array}{cc}
2&-1\\[4pt]
1&0
\end{array}
\right)^{a}=
\left(
\begin{array}{cc}
a+1&-a\\[4pt]
a&-(a-1)
\end{array}
\right).
$$
Then a direct computation leads to 
$M(1,\underbrace{2,\ldots,2}_{a},1,1)=-L^{a}$.
Hence the lemma.
\end{proof}

{\it Proof of Proposition~\ref{NegRegProp}.}
Since $M(1,1,1)=-\Id$, one gets from Lemma~\ref{TransMat}
$$
R^{a_i}L^{a_{i+1}}=-M(a_i+1,\underbrace{2,\ldots,2}_{a_{i+1}},1,1).
$$
Formula~(\ref{NegRegMat}) then follows from~(\ref{RegMatGenEq}) and the simple relation $M(2,1,1,a+1)=-M(a+2)$.

Proposition~\ref{NegRegProp} is proved.
\\

Finally, we observe that
the last three coefficients in~(\ref{NegRegMat}) are~$(2,1,1)$, and can be removed using
the equality $M(2,1,1)=-R$. So that one gets
\begin{equation}
\label{HZCorEq}
M^+(a_1,\ldots,a_{2m})=
M\big(a_1+1,\underbrace{2,\ldots,2}_{a_2-1},\,
a_3+2,\underbrace{2,\ldots,2}_{a_4-1},\ldots,
a_{2m-1}+2,\underbrace{2,\ldots,2}_{a_{2m}-1}\big)R.
\end{equation}
According to Proposition \ref{MatConv} the first column of the matrices from the right-hand-side and from the left-hand-side 
gives the rational ~$\frac{r}{s}$. Therefore, this establishes formula~\eqref{HZRegEqt} and
 the relation~\eqref{NegRegMatTheEq}.

\subsection{Converting the conjugacy classes in $\PSL(2,\Z)$}\label{HZCCSec}

We obtain it as a corollary of Proposition~\ref{NegRegProp}.

\begin{cor}
\label{HZCor}
The matrix $M^+(a_1,\ldots,a_{2m})$ is conjugacy equivalent
to the matrix
$$
M\big(a_1+2,\underbrace{2,\ldots,2}_{a_2-1},\,
a_3+2,\underbrace{2,\ldots,2}_{a_4-1},\ldots,
a_{2m-1}+2,\underbrace{2,\ldots,2}_{a_{2m}-1}\big).
$$
\end{cor}

\begin{proof}
This statement immediately follows
from~\eqref{HZCorEq} using conjugation by~$R$.
\end{proof}

The integers 
\begin{equation}
\label{ciCCSL}
(c_1,\ldots,c_k)=
\big(a_1+2,\underbrace{2,\ldots,2}_{a_2-1},\,
a_3+2,\underbrace{2,\ldots,2}_{a_4-1},\ldots,
a_{2m-1}+2,\underbrace{2,\ldots,2}_{a_{2m}-1}\big)
 \end{equation}
appearing in the above formula were used to describe
the conjugacy classes of $\PSL(2,\Z)$; see~\cite[p.91]{Zag} and provide interesting characteristics of
the quadratic irrationalities.

\begin{ex}
\label{ExT75BisBis}
Let us go back to Example~\ref{ExT75Bis} that treats the case of
the rational $\frac{r}{s}=\frac{7}{5}$.
Applying~(\ref{HZCorEq}), we get that $M^+(1,2,1,1)$ is conjugacy equivalent to
$$
M(3,2,3)=
\left(
\begin{array}{cc}
12&-5\\[4pt]
5&-2
\end{array}
\right).
$$
\end{ex}

\subsection{Appearance of the equation $M(c_1,\ldots,c_n)=-\Id$}\label{CoCoCFSec}
It turns out that, taking into account all the coefficients
$(c_1,\ldots,c_n)$ of the triangulation~$\mathbb{T}_{r/s}$ 
(and not only $(c_1,\ldots,c_k)$ as we did before), one obtains
the negative of the identity matrix.
The following statement can be found in~\cite{BR}.

\begin{prop}
\label{CoCoFac}
One has $M(c_1,\ldots,c_n)=-\Id.$
\end{prop}

\begin{proof}
Rewrite
$$
M(c_1,\ldots,c_n)=
M(c_1,\ldots,c_k)\,M(1)\,
M(c_{k+2},\ldots,c_{n-1})\, M(1),
$$
then~\eqref{HZCorEq} together with Corollary~\ref{CorRot} and Proposition~\ref{NegRegTheorem} imply
$$
\begin{array}{rcl}
M(c_1,\ldots,c_n)&=&
M^+(a_1,\ldots,a_{2m})\,R^{-1}M(1)\,M^+(a_{2m},\ldots,a_{1})\,R^{-1}M(1)\\[4pt]
&=&
M^+(a_1,\ldots,a_{2m})\,S\,M^+(a_{2m},\ldots,a_{1})\,S,
\end{array}
$$
where $S$ is as in~\eqref{RLSEq}.
Since $M^+(a_{2m},\ldots,a_{1})={M^+(a_1,\ldots,a_{2m})}^t$,
we conclude using the fact that
$ASA^tS=-\Id$ for all $A~\in\SL(2,\Z)$.
\end{proof}

Every rational number~$\frac{r}{s}$ thus corresponds to a solution of the
equation $M(c_1,\ldots,c_n)=-\Id$.
This equation will be important in the sequel for two reasons.

Firstly, the equation $M(c_1,\ldots,c_n)=-\Id$ makes sense and remains an interesting equation
in general, when there is no particular rational number and the corresponding continued fraction.
Expanding a rational in a continued fraction
$\frac{r}{s}=\llbracket{}c_1,\ldots,c_k\rrbracket$,
we always assumed $c_i\geq2$.
This assumption makes the expansion unique.
Allowing $c_i=1$ for some $i$, one faces two difficulties:
the expansion is no more unique (there is an infinite number of them),
and furthermore, the continued fraction may not be well-defined
(the denominators may vanish).
It turns out that considering the matrices $M(c_1,\ldots,c_k)$
with $c_i\geq1$ removes these difficulties.

Secondly, we will study the presentation of elements of the group
$\SL(2,\Z)$ (and $\PSL(2,\Z)$) in the form $A=M(c_1,\ldots,c_n)$
for some positive integers $c_i$.
Therefore, it will be important to know
the relations leading to different presentations of the same element.

\subsection{The semigroup $\Gamma$}

The matrices $M^+(a_1,\ldots,a_{2m})$ of regular continued fractions do not represent
arbitrary elements of~$\PSL(2,\Z)$.

\begin{defn}
{\rm
The semigroup $\Gamma\subset\SL(2,\Z)$ consists of the elements
$M^+(a_1,\ldots,a_{2m})$
where~$a_i$ are positive integers.
}
\end{defn}

As mentioned in Proposition~\ref{MaTCompProp}, $\Gamma$ is generated by the matrices~$R$ and~$L$.
It consists of the matrices with positive entries satisfying the following conditions:
$$
 \Gamma=
\left\{ \begin{pmatrix}
a&b\\
c&d
\end{pmatrix}
\in \SL_2(\Z)
\right.
\left|
\begin{array}{l}
a\geq{}b\geq{}d>0,\\[2pt]
a\geq{}c\geq{}d>0
\end{array}
\right\}.
$$

The semigroup $\Gamma$ is the main character of a wealth of different problems
of number theory, dynamics, combinatorics, etc.
It was studied by many authors from different viewpoints;
see~\cite{Aig,McM,Boc,BK}
and references therein.

This is one of the motivations for a systematic study of
the matrices $M(c_1,\ldots,c_n)$ which is one of the main subjects of this paper.

\section{Solving the equation $M(c_1,\ldots,c_n)=\pm\Id$}\label{CoCoIntSec}

In this section we describe all positive integer solution of the two equations
$$
M(c_1,\ldots,c_n)=-\Id,
\qquad\hbox{and}\qquad
M(c_1,\ldots,c_n)=\Id,
$$
for the matrices~\eqref{SLEq}.
Recall that matrices $M(c_1,\ldots,c_n)$ with $c_i\geq1$, satisfying $M(c_1,\ldots,c_n)=-\Id$,
arose from continued fractions, see Section~\ref{CoCoCFSec}.
The equation $M(c_1,\ldots,c_n)=\Id$ is quite different but also relevant.

One motivation for considering solutions with arbitrary positive integers $c_i\geq1$
is related to the observation that positive integers usually count interesting combinatorial objects.
The solutions we classify in this section are given in terms of polygon dissections:
triangulations and also more general ``$3d$-dissections''  of $n$-gons.
Another motivation is to extend most of the results 
and ideas of Section~\ref{TriangCFSec} from continued fractions
to arbitrary solutions of the equation $M(c_1,\ldots,c_n)=\pm\Id$.
Our third motivation is related to a
more general study (see Section~\ref{DecSec}) of decomposition of an arbitrary element $A\in\PSL(2,\Z)$ in the form
$A=M(c_1,\ldots,c_n)$.
Solutions of the above equations describe relations in such a decomposition. 

Let us also mention that
equation $M(c_1,\ldots,c_n)=-\Id$ considered over~$\C$ defines an interesting algebraic variety
closely related to the classical moduli space~$\cM_{0,n}$ of configurations of points in the projective line.
Therefore, positive integer solutions of
this equation correspond to a class of rational points of~$\cM_{0,n}$; see~\cite{SVRS}.
We do not consider geometric applications in the present paper.

\subsection{Conway and Coxeter totally positive solutions}\label{QuidSSec}
A classical theorem of Conway and Coxeter~\cite{CoCo} describes 
a particular class of solutions of the equation
\begin{equation}
\label{CoCEqn}
M(c_1,\ldots,c_n)=-\Id.
\end{equation}
More importantly, this theorem
relates this equation to combinatorics.

The following notion is the most important ingredient of the theory.

\begin{defn}
\label{QuidDef}
{\rm
(a)
Given a triangulation of a convex $n$-gon by non-crossing diagonals,
its {\it quiddity} is the
(cyclically ordered) $n$-tuple of positive integers, $(c_1,\ldots,c_n)$,
counting the number of triangles adjacent to the vertices.

(b)
Given an $n$-tuple of positive integers, $(c_1,\ldots,c_n)$,
we consider the following sequence of rational numbers,
or infinity:
$$
\frac{r_i}{s_i}:=
\llbracket{}c_1,\ldots,c_i\rrbracket,
$$
for $1\leq{}i\leq{}n$.}

\end{defn}

For example, the coefficients $c_i$ of a negative continued fraction
of a rational number~$\frac{r}{s}$ is
a part of the quiddity of the triangulation~$\mathbb{T}_{r/s}$,
and the rationals~$\frac{r_i}{s_i}$ are its convergents;
see Section~\ref{CombIntCFSec}.
Of course, for continued fractions, the denominator of~$\frac{r_i}{s_i}$
cannot vanish.

\begin{defn}
\label{TPSDef}
The class of solutions of~\eqref{CoCEqn} satisfying the condition
\begin{equation}
\label{ToPoEq}
\frac{r_i}{s_i}>0,
\end{equation} 
for all $i\leq{}n-3$,
will be called {\it totally positive}.
\end{defn}

We will see in Section~\ref{MoyTSSec} that the above condition of total positivity is equivalent to the assumption that
$c_1+c_2+\cdots+c_n=3n-6$.

The Conway and Coxeter theorem~\cite{CoCo}
establishes a one-to-one correspondence between totally positive
solutions of~(\ref{CoCEqn}) and triangulations of the $n$-gon, via the notion of quiddity that uniquely determines the triangulation.

\begin{thm}[\cite{CoCo}]
\label{CoCoThm}
(i) 
The quiddity of a  triangulated $n$-gon is a totally positive solution of~(\ref{CoCEqn}).

(ii)
A totally positive solution of~(\ref{CoCEqn}) is the quiddity of
a triangulated $n$-gon.
\end{thm}

We do not dwell on the detailed proof of this classical result.
For a simple complete proof of Theorem~\ref{CoCoThm} see~\cite{BR,Hen},
and also~\cite{Ovs}.
The idea of the proof consists of three observations.

1)
An $n$-tuple of integers $(c_1,\ldots,c_n)$ satisfying~(\ref{CoCEqn}) 
must contain $c_i=1$ for some $i$.
Otherwise, for any sequence of integers $(v_i)_{i\in\Z}$
satisfying the linear recurrence~$v_{i+1}=c_iv_i-v_{i-1}$, with the initial conditions
$(v_0,v_1)=(0,1)$, one has: $v_{i+1}>v_i$.
Therefore, the sequence~$(v_i)_{i\in\Z}$ cannot be periodic.
This contradicts the equation~$M(c_1,\ldots,c_n)\left(\begin{array}{c}1\\0\end{array}\right)
=\pm\left(\begin{array}{c}1\\0\end{array}\right)$.

2)
The total positivity condition~\eqref{ToPoEq} implies that,
whenever $c_i=1$ for some $i$, the two neighbors $c_{i-1},c_{i+1}$
must be greater or equal to~$2$.
Indeed, for two consecutive~$1$'s, if~$(c_i,c_{i+1})=(1,1)$, one has
$v_{i+2}=v_{i+1}-v_i=v_i-v_{i-1}-v_i=-v_{i-1}$.

3)
The ``local surgery'' operation 
\begin{equation}
\label{FirstSurM}
(c_1,\ldots,c_{i-1},\,1,c_{i+1},\ldots,c_n)
\to(c_1,\ldots,c_{i-1}-1,\,c_{i+1}-1,\ldots,c_n)
\end{equation}
is then well-defined.
It decreases $n$ by $1$, and does not change the matrix~\eqref{SLEq}.

Indeed,
$M(c_1,\ldots,c_{i-1},\,1,c_{i+1},\ldots,c_n)=
M(c_1,\ldots,c_{i-1}-1,\,c_{i+1}-1,\ldots,c_n)$
since
$$
\begin{array}{rcl}
\left(
\begin{array}{cc}
c+1&-1\\[4pt]
1&0
\end{array}
\right)
\left(
\begin{array}{cc}
1&-1\\[4pt]
1&0
\end{array}
\right)
\left(
\begin{array}{cc}
c'+1&-1\\[4pt]
1&0
\end{array}
\right)&=&
\left(
\begin{array}{cc}
cc'-1&-c'\\[4pt]
c&-1
\end{array}
\right)\\[15pt]
&=&
\left(
\begin{array}{cc}
c&-1\\[4pt]
1&0
\end{array}
\right)
\left(
\begin{array}{cc}
c'&-1\\[4pt]
1&0
\end{array}
\right).
\end{array}
$$

One then proceeds by induction on $n$,
the induction step consists of cutting an exterior triangle of a given triangulation, that corresponds
to the operation~(\ref{FirstSurM}) on the quiddity.

\begin{ex}
\label{FirstEx}
The sequence $(2,2,2,5,4,2,2,1,4,2,4,1,3,2,5,1)$ is a solution of~(\ref{CoCEqn})
because it is the quiddity of the following triangulation of a hexadecagon:
$$
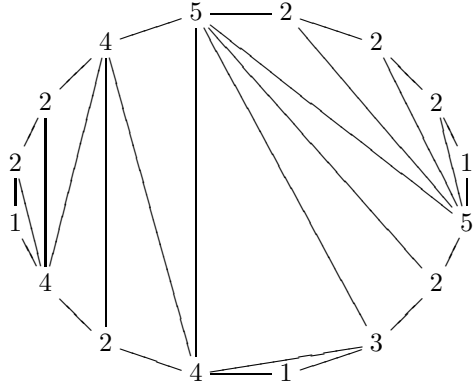
\xymatrix @!0 @R=0.40cm @C=0.40cm
 {
&&&&&&5\ar@{-}[dddddddddddd]\ar@{-}[llld]\ar@{-}[rrrrrrddddddddddd]
\ar@{-}[rrrrrrrrddddddddd]\ar@{-}[rrrrrrrrrddddddd]\ar@{-}[rrr]&&&
2\ar@{-}[rrrd]\ar@{-}[rrrrrrddddddd]
\\
&&&4\ar@{-}[lldd]\ar@{-}[lldddddddd]\ar@{-}[rrrddddddddddd]\ar@{-}[dddddddddd]&&&&&&&&& 
2\ar@{-}[rrdd]\ar@{-}[rrrdddddd]\\
\\
&2\ar@{-}[ldd]\ar@{-}[dddddd]&&&&&&&&&&&&&2\ar@{-}[rdd]\ar@{-}[rdddd]\\
\\
2\ar@{-}[dd]\ar@{-}[rdddd]&&&&&&&&&&&&&&&1\ar@{-}[dd]\\
\\
1\ar@{-}[rdd]&&&&&&&&&&&&&&&5\ar@{-}[ldd]\\
\\
&4\ar@{-}[rrdd]&&&&&&&&&&&&&2\ar@{-}[lldd]\\
\\
&&&2\ar@{-}[rrrd]&&&&&&&&& 3\ar@{-}[llld]\\
&&&&&&4\ar@{-}[rrr]\ar@{-}[rrrrrru]&&&1
}
$$
\end{ex}

\begin{rem}
a)
The meaning of total positivity will be explained in
Sections~\ref{ConstrWalkFSec} and~\ref{PtoSec}.

b)
The Conway and Coxeter theorem is initially formulated in terms of {\it frieze patterns}.
This notion, due to Coxeter~\cite{Cox}, became popular mainly because its relations
to cluster algebras; see~\cite{CaCh}.
Frieze patterns also play an important role in such areas as
quiver representations, differential geometry, discrete integrable systems
(for a survey; see~\cite{Sop}).
\end{rem}

\subsection{The complete set of solutions: $3d$-dissections}\label{MoySec}
It turns out that, to classify all the solutions
(with no total positivity condition), it is natural to solve simultaneously
the equations 
$$
M(c_1,\ldots,c_n)=-\Id
\qquad\hbox{and}\qquad
M(c_1,\ldots,c_n)=\Id.
$$
This classification led to the following combinatorial notion.

\begin{defn}
\label{3dDef}
{\rm
(i)
A $3d$-{\it dissection} is a partition of a convex $n$-gon
into sub-polygons by means of pairwise non-crossing diagonals, 
such that the number of vertices of every sub-polygon is a multiple of $3$.

(ii)
The quiddity of a $3d$-dissection of an $n$-gon is defined, similarly to the case of a triangulation,
as a cyclically ordered sequence $(c_1,\ldots,c_n)$ of positive integers counting sub-polygons adjacent to every vertex.}

\end{defn}

The following theorem was proved in~\cite{Ovs}.

\begin{thm}[\cite{Ovs}]
\label{SecondMainThm}
(i)
The quiddity of a $3d$-dissection of an $n$-gon satisfies
$M(c_1,\ldots,c_n)=\pm\Id$.

(ii)
Conversely,
every solution of the equation $M(c_1,\ldots,c_n)=\pm\Id$ with positive integers~$c_i$ is the quiddity of 
a $3d$-dissection of an $n$-gon.
\end{thm}

Similarly to Theorem~\ref{CoCoThm}, the proof uses induction on $n$.
The idea is as follows. 
Besides the operations~(\ref{FirstSurM}), one needs 
another type of ``local surgery'' operations.
These operations remove two consecutive $1$'s:
\begin{equation}
\label{SecSurM}
(c_1,\ldots,c_{i-1},\,c_i,\,1,\,\,1,\,c_{i+3},\,c_{i+4},\ldots,c_n)
\to(c_1,\ldots,c_{i-1},\,c_i+c_{i+3}-1,\,c_{i+4},\ldots,c_n).
\end{equation}
Such an operation decreases $n$ by $3$ and changes the sign of the matrix
$M(c_1,\ldots,c_n)$.
Indeed, 
$$
\left(
\begin{array}{cc}
c'&-1\\[4pt]
1&0
\end{array}
\right)
\left(
\begin{array}{cc}
1&-1\\[4pt]
1&0
\end{array}
\right)^2
\left(
\begin{array}{cc}
c''&-1\\[4pt]
1&0
\end{array}
\right)\,=\,
\left(
\begin{array}{cc}
1-c'-c''&1\\[4pt]
-1&0
\end{array}
\right).
$$

\begin{ex}
\label{SimpExQuid}
Simple examples of $3d$-dissections different from triangulations are:
$$
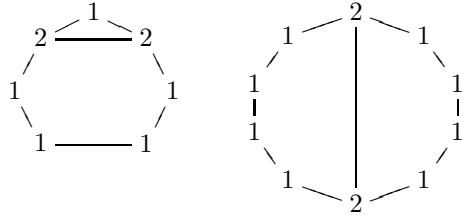
\xymatrix @!0 @R=0.35cm @C=0.35cm
{
&&&1\ar@{-}[rrd]\ar@{-}[lld]
\\
&2\ar@{-}[ldd]\ar@{-}[rrrr]&&&& 2\ar@{-}[rdd]&\\
\\
1\ar@{-}[rdd]&&&&&& 1\ar@{-}[ldd]\\
\\
&1\ar@{-}[rrrr]&&&&1
}
\qquad
\xymatrix @!0 @R=0.32cm @C=0.45cm
 {
&&&2\ar@{-}[dddddddd]\ar@{-}[lld]\ar@{-}[rrd]&
\\
&1\ar@{-}[ldd]&&&& 1\ar@{-}[rdd]\\
\\
1\ar@{-}[dd]&&&&&&1\ar@{-}[dd]\\
\\
1\ar@{-}[rdd]&&&&&&1\ar@{-}[ldd]\\
\\
&1\ar@{-}[rrd]&&&& 1\ar@{-}[lld]\\
&&&2&
}
$$
Their quiddities are solutions of~(\ref{SLCoCoEq}).
More precisely, 
$$
M(1,1,2,1,2,1,1)=\Id,
\qquad
M(1,1,2,1,1,1,1,2,1,1)=-\Id.
$$
\end{ex}

\begin{rem}
\label{NoNUnRem}
Theorem~\ref{SecondMainThm}
does not imply a one-to-one correspondence
between solutions of~(\ref{SLCoCoEq}) and $3d$-dissections.
Moreover, such a correspondence does not exist.
Indeed, the quiddity of a $3d$-dissection does not characterize it.
This means that different $3d$-dissections
may correspond to the same quiddity.
For instance, the following different $3d$-dissections of the octagon
$$
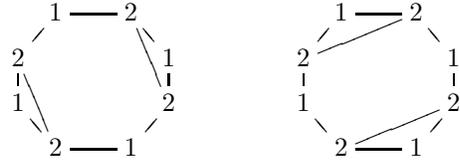
\xymatrix @!0 @R=0.30cm @C=0.5cm
 {
&1\ar@{-}[ldd]\ar@{-}[rr]&& 2\ar@{-}[rdd]\\
\\
2\ar@{-}[dd]\ar@{-}[rdddd]&&&&1\ar@{-}[dd]\\
\\
1\ar@{-}[rdd]&&&&2\ar@{-}[ldd]\ar@{-}[luuuu]\\
\\
&2\ar@{-}[rr]&& 1
}
\qquad\qquad
\xymatrix @!0 @R=0.30cm @C=0.5cm
 {
&1\ar@{-}[ldd]\ar@{-}[rr]&& 2\ar@{-}[rdd]\\
\\
2\ar@{-}[dd]\ar@{-}[rrruu]&&&&1\ar@{-}[dd]\\
\\
1\ar@{-}[rdd]&&&&2\ar@{-}[ldd]\ar@{-}[llldd]\\
\\
&2\ar@{-}[rr]&& 1
}
$$
have the same quiddity.
This observation is due to Alexey Klimenko.
\end{rem}

To elucidate the statement of Theorem~\ref{SecondMainThm},
let us separate the cases of $-\Id$ and $\Id$.

\begin{cor}
\label{PlusMinProp}
Given a $3d$-dissection of an $n$-gon, its quiddity~$(c_1,\ldots,c_n)$ satisfies
$M(c_1,\ldots,c_n)=-\Id$ if and only if the number of subpolygons 
with an even number of vertices is even.
\end{cor}

\subsection{The total sum $c_1+\cdots+c_n$}\label{MoyTSSec}
An interesting characteristics of a solution is
the  total sum of the coefficients $c_i$.

Theorem~\ref{SecondMainThm} implies that the value
$c_1+\cdots+c_n=3n-6$ is maximal.
Note that, for a totally positive solution, the sum of $c_i$'s is equal to $3n-6$
which is three times the number of triangles in a triangulation
of an $n$-gon.

\begin{cor}
\label{MaxProp}
Positive integer solutions of  the equation
$M(c_1,\ldots,c_n)=\pm\Id$ always satisfy
$$
c_1+c_2+\cdots+c_n\leq3n-6.
$$
\end{cor}

The total sum can be expressed in terms of the subpolygons
of the corresponding $3d$-dissection.
The following statement is a combination of Corollary~2.3 and Proposition~3.1 of~\cite{Ovs},
we do not dwell on the proof here.

\begin{prop}
\label{TSCorProp}
The total sum of $c_i$'s in the quiddity of a $3d$-dissection of an $n$-gon is
$$
c_1+c_2+\cdots+c_n=
3n-6\sum_{k\leq\left[\frac{n}{3}\right]}
\left(k-1\right)
N_k-6,
$$
where $N_k$ is the number of $3k$-gons in the $3d$-dissection.
\end{prop}

It follows that
the total sum of $c_i$'s can vary by multiples of $6$, and
the sign on the right-hand side of~$M(c_1,\ldots,c_n)=\pm\Id$ alternates.

\begin{cor}
The solutions of the equation
$M(c_1,\ldots,c_n)=\pm\Id$ can be ranged by levels:
\begin{equation}
\label{TotSumDropEq}
\begin{array}{rclr}
c_1+c_2+\cdots+c_n &=& 3n-6, &(-\Id)\\[2pt]
&=& 3n-12, &(\Id)\\[2pt]
&=& 3n-18,&(-\Id)\\
&&\ldots&
\end{array}
\end{equation}
\end{cor}

\section{Walks on the Farey graph}\label{FarSec}

In this section we show that every solution of
the equation $M(c_1,\ldots,c_n)=\pm\Id$ admits an embedding into the Farey tessellation.
This is a generalization of the construction from Section~\ref{TrsFarSec}.

In particular, a totally positive solution 
corresponding to a triangulation of the $n$-gon,
defines a monotonously decreasing walk from~$\frac{1}{0}$ to~$\frac{0}{1}$.
This is an $n$-cycle in the Farey graph that we refer to as a
``Farey $n$-gon''.
The Farey tessellation then induces a triangulation which
coincides with the initial triangulation.

A more general solution corresponding to a $3d$-dissection of an $N$-gon
defines (an oriented) walk along a certain Farey $n$-gon, where $n<N$.
Every such walk is an $N$-cycle, and we show that
the quiddity of the $3d$-dissection of the $N$-gon can be recovered
from the triangulation of the Farey $n$-gon.

\subsection{Solutions of $M(c_1,\ldots,c_n)=-\Id$ and $n$-cycles in the Farey graph}\label{ConstrWalkFSec}

We use the following combinatorial data in the Farey graph.

\begin{defn}
(i) 
An {\it $n$-cycle} in the Farey graph is a sequence $(v_{i})_{i \in\Z}$ of vertices
 (with the cyclic order convention $v_{i+n}=v_{i}$),
such that $v_{i-1}$ and $v_{i}$ are connected by an edge for all $i$.

(ii)
We call a {\it Farey $n$-gon} every $n$-cycle in the Farey graph such that 
$$
v_{0}=\frac10,
\qquad
 v_{n-1}=\frac01,
 \qquad\hbox{and}\qquad
 v_{i-1}>v_{i},
 $$
 for all $i=1, \ldots, n-1$.
\end{defn}

\begin{ex}
\label{DiffEx}
The sequence
$\left\{\frac10,\frac21,\frac32,\frac43,\frac{17}{13},
\frac{64}{49},\frac{111}{85},\frac{158}{121},\frac{47}{36},\frac{30}{23},
\frac{13}{10},\frac{22}{17},\frac97,\frac54,\frac11, \frac01\right\}$
is a Farey hexadecagon:
\begin{center}
\psscalebox{1.0 1.0} 
{\psset{unit=0.6cm}
\begin{pspicture}(0,-5.5125)(20.957693,5.5125)
\psdots[linecolor=black, dotsize=0.16](1.6788461,-4.9075)
\psdots[linecolor=black, dotsize=0.16](9.678846,-4.9075)
\psdots[linecolor=black, dotsize=0.16](0.07884616,-4.9075)
\psdots[linecolor=black, dotsize=0.16](11.278846,-4.9075)
\psdots[linecolor=black, dotsize=0.16](3.2788463,-4.9075)
\psdots[linecolor=black, dotsize=0.16](4.478846,-4.9075)
\psdots[linecolor=black, dotsize=0.16](10.478847,-4.9075)
\psdots[linecolor=black, dotsize=0.16](6.878846,-4.9075)
\psdots[linecolor=black, dotsize=0.16](8.078846,-4.9075)
\psdots[linecolor=black, dotsize=0.16](8.878846,-4.9075)
\psdots[linecolor=black, dotsize=0.16](12.478847,-4.9075)
\psdots[linecolor=black, dotsize=0.16](5.6788464,-4.9075)
\psdots[linecolor=black, dotsize=0.16](13.678846,-4.9075)
\psdots[linecolor=black, dotsize=0.16](15.278846,-4.9075)
\psdots[linecolor=black, dotsize=0.16](16.878845,-4.9075)
\psdots[linecolor=black, dotsize=0.16](20.878845,-4.9075)
\psarc[linecolor=black, linewidth=0.08, dimen=outer](6.2788463,-4.9075){0.6}{0.0}{180.0}
\psarc[linecolor=black, linewidth=0.08, dimen=outer](7.478846,-4.9075){0.6}{0.0}{180.0}
\psarc[linecolor=black, linewidth=0.08, dimen=outer](8.478847,-4.9075){0.4}{0.0}{180.0}
\psarc[linecolor=black, linewidth=0.08, dimen=outer](9.278846,-4.9075){0.4}{0.0}{180.0}
\psarc[linecolor=black, linewidth=0.08, dimen=outer](10.078846,-4.9075){0.4}{0.0}{180.0}
\psarc[linecolor=black, linewidth=0.08, dimen=outer](10.878846,-4.9075){0.4}{0.0}{180.0}
\psarc[linecolor=black, linewidth=0.08, dimen=outer](11.878846,-4.9075){0.6}{0.0}{180.0}
\psarc[linecolor=black, linewidth=0.08, dimen=outer](13.078846,-4.9075){0.6}{0.0}{180.0}
\psarc[linecolor=black, linewidth=0.08, dimen=outer](0.87884617,-4.9075){0.8}{0.0}{180.0}
\psarc[linecolor=black, linewidth=0.08, dimen=outer](2.478846,-4.9075){0.8}{0.0}{180.0}
\psarc[linecolor=black, linewidth=0.08, dimen=outer](3.8788462,-4.9075){0.6}{0.0}{180.0}
\psarc[linecolor=black, linewidth=0.08, dimen=outer](5.078846,-4.9075){0.6}{0.0}{180.0}
\psarc[linecolor=black, linewidth=0.08, dimen=outer](10.478847,-4.9075){10.4}{0.0}{180.0}
\psarc[linecolor=black, linewidth=0.08, dimen=outer](14.478847,-4.9075){0.8}{0.0}{180.0}
\psarc[linecolor=black, linewidth=0.08, dimen=outer](16.078846,-4.9075){0.8}{0.0}{180.0}
\psarc[linecolor=black, linewidth=0.08, dimen=outer](18.878845,-4.9075){2.0}{0.0}{180.0}
\psarc[linecolor=black, linewidth=0.02, dimen=outer](9.678846,-4.9075){0.8}{0.0}{180.0}
\psarc[linecolor=black, linewidth=0.02, dimen=outer](5.6788464,-4.9075){1.2}{0.0}{180.0}
\psarc[linecolor=black, linewidth=0.02, dimen=outer](10.078846,-4.9075){1.2}{0.0}{180.0}
\psarc[linecolor=black, linewidth=0.02, dimen=outer](10.678846,-4.9075){1.8}{0.0}{180.0}
\psarc[linecolor=black, linewidth=0.02, dimen=outer](11.278846,-4.9075){9.6}{0.0}{180.0}
\psarc[linecolor=black, linewidth=0.02, dimen=outer](9.278846,-4.9075){7.6}{0.0}{180.0}
\psarc[linecolor=black, linewidth=0.02, dimen=outer](8.478847,-4.9075){6.8}{0.0}{180.0}
\psarc[linecolor=black, linewidth=0.02, dimen=outer](7.6788464,-4.9075){6.0}{0.0}{180.0}
\psarc[linecolor=black, linewidth=0.02, dimen=outer](8.478847,-4.9075){5.2}{0.0}{180.0}
\psarc[linecolor=black, linewidth=0.02, dimen=outer](9.078846,-4.9075){4.6}{0.0}{180.0}
\psarc[linecolor=black, linewidth=0.02, dimen=outer](10.278846,-4.9075){3.4}{0.0}{180.0}
\psarc[linecolor=black, linewidth=0.02, dimen=outer](10.278846,-4.9075){2.2}{0.0}{180.0}
\psarc[linecolor=black, linewidth=0.02, dimen=outer](9.678846,-4.9075){2.8}{0.0}{180.0}
\rput(0.07884616,-5.4075){$\frac01$}
\rput(1.6788461,-5.4075){$\frac11$}
\rput(3.2788463,-5.4075){$\frac54$}
\rput(4.478846,-5.4075){$\frac97$}
\rput(20.878845,-5.4075){$\frac10$}
\rput(16.878845,-5.4075){$\frac21$}
\rput(15.278846,-5.4075){$\frac32$}
\rput(13.678846,-5.4075){$\frac43$}
\rput(12.478847,-5.4075){$\frac{17}{13}$}
\rput(11.278846,-5.4075){$\frac{64}{49}$}
\rput(10.478847,-5.4075){$\frac{111}{85}$}
\rput(9.678846,-5.4075){$\frac{158}{121}$}
\rput(8.878846,-5.4075){$\frac{47}{36}$}
\rput(8.078846,-5.4075){$\frac{30}{23}$}
\rput(6.878846,-5.4075){$\frac{13}{10}$}
\rput(5.6788464,-5.4075){$\frac{22}{17}$}
\end{pspicture}
}
\end{center}
\end{ex}

Let $(c_1,\ldots,c_n)$ be a set of positive integers such that $M(c_1,\ldots,c_n)=-\Id$.
Our next goal is to define the corresponding $n$-cycle in the Farey graph.

We define a sequence of $n$ vertices in the Farey graph
$\left(\frac{r_0}{s_0},\ldots,\frac{r_{n-1}}{s_{n-1}}\right)$  that starts with~$\frac{1}{0}$
and ends with~$\frac{0}{1}$\, with the following recurrence relations:
\begin{equation}
\label{CycEqBis}
\left\{
\begin{array}{rcl}
r_i &:=& c_ir_{i-1}-r_{i-2}\\
s_i &:=& c_is_{i-1}-s_{i-2}.
\end{array}\right.
\end{equation}
and the initial conditions
$\frac{r_{-1}}{s_{-1}}=\frac{0}{-1}$, $\frac{r_0}{s_0}=\frac10$.
We get a sequence of the form
\begin{equation}
\label{CycEq}
\left(\frac{r_0}{s_0},\ldots,\frac{r_{n-1}}{s_{n-1}}\right)=
\left(
\frac{1}{0},\quad
\frac{c_1}{1},\quad
\frac{c_1c_2-1}{c_2},\quad
\frac{c_1c_2c_3-c_1-c_2}{c_2c_3-1},\;
\ldots\;,
\quad
\frac{0}{1}
\right);
\end{equation}
The fact that~$\frac{r_{n-1}}{s_{n-1}}=\frac{0}{1}$
follows from the relation $M(c_1,\ldots,c_n)=-\Id$.
Indeed, inductively one obtains
$$
M(c_1,\ldots,c_i)=
\left(
\begin{array}{cc}
r_i&-r_{i-1}\\[4pt]
s_i&-s_{i-1}
\end{array}
\right),
$$
where
$$
\frac{r_i}{s_i}=\llbracket{}c_1,\ldots,c_i\rrbracket{}.
$$
is the sequence of convergents of the negative continued fraction 
$\llbracket{}c_1,\ldots,c_n\rrbracket{}$.
In particular, $-\Id=M(c_1,\ldots,c_n)=\left(
\begin{array}{cc}
r_n&-r_{n-1}\\[4pt]
s_n&-s_{n-1}
\end{array}
\right),$
so $r_{n-1}=0$ and $s_{n-1}=1$,
and furthermore $r_n=-1=-r_0$ and $s_n=0=s_0$.
This implies that the sequences defined by~\eqref{CycEqBis} are $n$-antiperiodic,
and one obtains a sequence
$\left(\frac{r_i}{s_i}\right)_{i\in\Z}$, such that
$$
\frac{r_{i+n}}{s_{i+n}}=\frac{-r_i}{-s_i}.
$$ 

\begin{prop}
\label{nCycFFac}
(i)
If $(c_1,\ldots,c_n)$ is a positive solution of $M(c_1,\ldots,c_n)=-\Id$, then
the sequence~\eqref{CycEq} is an $n$-cycle in the Farey graph.

(ii) If $(c_1,\ldots,c_n)$ is a totally positive solution, then 
the sequence~\eqref{CycEq} is a Farey $n$-gon.
\end{prop}
\begin{proof}
Part~(i).
For every pair of sequences, $(r_i)$ and $(s_i)$, satisfying the linear recurrence~\eqref{CycEqBis}
one obtains constant $2\times2$ determinants
$$
\det
\left(
\begin{array}{cc}
r_i&r_{i-1}\\[4pt]
s_i&s_{i-1}
\end{array}
\right)=
\det
\left(
\begin{array}{cc}
-r_{i-2}&r_{i-1}\\[4pt]
-s_{i-2}&s_{i-1}
\end{array}
\right)=
\det
\left(
\begin{array}{cc}
r_{i-1}&r_{i-2}\\[4pt]
s_{i-1}&s_{i-2}
\end{array}
\right)=\ldots =
\det
\left(
\begin{array}{cc}
r_{0}&r_{-1}\\[4pt]
s_{0}&s_{-1}
\end{array}
\right)=-1.
$$ 
Therefore,~$\frac{r_i}{s_i}$ and~$\frac{r_{i-1}}{s_{i-1}}$ are connected by an edge.

Part~(ii).
The positivity condition \eqref{ToPoEq} and the relation $r_is_{i-1}-r_{i-1}s_i=-1$
imply $\frac{r_{i}}{s_{i}}<\frac{r_{i-1}}{s_{i-1}}$.

Hence the result.
\end{proof}

It turns out that every Farey $n$-gon gives rise to a totally positive solution, 
so that we can formulate the following statement. For details; see~\cite[Proposition~2.2.1]{SVS}.

\begin{cor}
\label{CoCoFacCor}
Totally positive solutions of the equation $M(c_1,\ldots,c_n)=-\Id$
are in one-to-one correspondence with Farey $n$-gons.
\end{cor}

\begin{rem}
The solution $(c_{1}, \ldots, c_{n})$ can be recovered from the $n$-gon using the 
 notion of {\it index} of a Farey sequence (this notion was defined and studied in~\cite{HS}).
Given an $n$-gon 
$
\left(
\frac{r_0}{s_0},\frac{r_1}{s_1},\ldots,\frac{r_{n-1}}{s_{n-1}}
\right)
$
in the Farey graph, its index
is the $n$-tuple of integers
$$
c_i:=\frac{r_{i-1}+r_{i+1}}{r_i}=
\frac{s_{i-1}+s_{i+1}}{s_i}.
$$
In our terms, the index is nothing else than the quiddity of the triangulation.
\end{rem}

\subsection{Farey $n$-gons and triangulations}\label{EMSec}
Classical properties on the Farey graph imply the following statement whose proof can be found
in~\cite{SVS}.

\begin{prop}
Every Farey $n$-gon is triangulated in the Farey tessellation. 
\end{prop}

In other words, the full subgraph of the Farey graph, containing the vertices of a Farey $n$-gon forms a triangulation of the $n$-gon.
Example \ref{DiffEx} gives an illustration of this statement.

Therefore from a totally positive solution of $M(c_1,\ldots,c_n)=-\Id$, one obtains two triangulations:
one given by Conway and Coxeter's correspondence (see Theorem~\ref{CoCoThm}, Part~(ii)), and 
the other one given by Farey $n$-gons, see Corollary \ref{CoCoFacCor}.

\begin{thm}[\cite{SVS}, Theorem~1]
\label{FarFac}
The Conway and Coxeter triangulation coincides with the Farey triangulation.
\end{thm}

We do not dwell on the proof here (see~\cite[Section~2.2]{SVS}).

The rational labels on the vertices of the triangulation can be recovered directly from the triangulation by the following combinatorial algorithm.
Note that this is the same algorithm as in Section~\ref{FSLabSec}, except for step (1),
and applied to arbitrary triangulations.

\begin{enumerate}
\item
In the triangulated $n$-gon,  label the vertex number $1$  
by~$\frac{1}{0}$ and the vertex number $n$ by $\frac{0}{1}$.
\item
Label all the vertices of the $n$-gon according to the rule:
Whenever two vertices of the same triangle have been assigned the rationals
$\frac{r'}{s'}$ and $\frac{r''}{s''}$, then the third vertex receives the label
$$
\frac{r'}{s'}\oplus\frac{r''}{s''}
:=\frac{r'+r''}{s'+s''}.
$$
\end{enumerate}

\begin{ex}
\label{DiffExBis}
Applying the above rule for the rational labelling on the hexadecagon of Example~\ref{FirstEx}, one 
obtains
$$
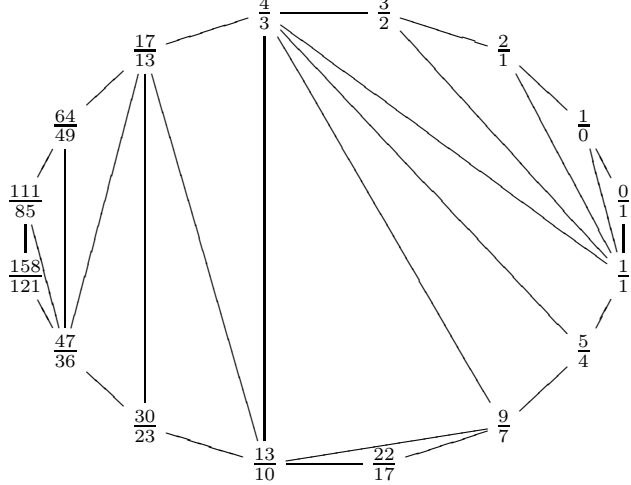
\xymatrix @!0 @R=0.50cm @C=0.53cm
 {
&&&&&&\frac43\ar@{-}[dddddddddddd]\ar@{-}[llld]\ar@{-}[rrrrrrddddddddddd]
\ar@{-}[rrrrrrrrddddddddd]\ar@{-}[rrrrrrrrrddddddd]\ar@{-}[rrr]&&&
\frac32\ar@{-}[rrrd]\ar@{-}[rrrrrrddddddd]
\\
&&&\frac{17}{13}\ar@{-}[lldd]\ar@{-}[lldddddddd]
\ar@{-}[rrrddddddddddd]\ar@{-}[dddddddddd]&&&&&&&&& 
\frac21\ar@{-}[rrdd]\ar@{-}[rrrdddddd]\\
\\
&\frac{64}{49}\ar@{-}[ldd]\ar@{-}[dddddd]&&&&&&&&&&&&&
\frac10\ar@{-}[rdd]\ar@{-}[rdddd]\\
\\
\frac{111}{85}\ar@{-}[dd]\ar@{-}[rdddd]&&&&&&&&&&&&&&&\frac01\ar@{-}[dd]\\
\\
\frac{158}{121}\ar@{-}[rdd]&&&&&&&&&&&&&&&\frac11\ar@{-}[ldd]\\
\\
&\frac{47}{36}\ar@{-}[rrdd]&&&&&&&&&&&&&\frac54\ar@{-}[lldd]\\
\\
&&&\frac{30}{23}\ar@{-}[rrrd]&&&&&&&&& \frac97\ar@{-}[llld]\\
&&&&&&\frac{13}{10}\ar@{-}[rrr]\ar@{-}[rrrrrru]&&&\frac{22}{17}
}
$$
which coincides with the Farey hexadecagon of Example~\ref{DiffEx}.
Note that
the triangulation induced from the Farey tessellation coincides with the initial triangulation in
Example~\ref{FirstEx}.
\end{ex}

Consider a cyclic permutation of $(c_1,\ldots,c_n)$.
This gives another solution of
$M(c_1,\ldots,c_n)=-\Id$.
Clearly the corresponding triangulations given by Theorem~\ref{CoCoThm},
Part~(ii) are related by a cyclic permutation of the vertices, i.e. a rotation.

\begin{prop}
The Farey $n$-gons corresponding to a cyclic permutation of
$(c_1,\ldots,c_n)$ are related by a cyclic permutation
modulo the action of $\SL(2,\Z)$ by linear-fractional transformations.
\end{prop}

This statement is proved in~\cite[Proposition 2.2.1]{SVS}.

\begin{ex}
Let $n=6$, and consider the totally positive solutions
$(3,1,3,1,3,1)$ and $(1,3,1,3,1,3)$.
They correspond to the Farey hexagons
$$
\left(
\frac{1}{0},\,\frac{3}{1},\,\frac{2}{1},\, \frac{3}{2},\,\frac{1}{1},\,\frac{0}{1}\
\right)
\qquad\hbox{and}\qquad
\left(
\frac{1}{0},\,\frac{1}{1},\,\frac{2}{3},\, \frac{1}{2},\,\frac{1}{3},\,\frac{0}{1}\
\right)
$$
and to the triangulated hexagons
$$
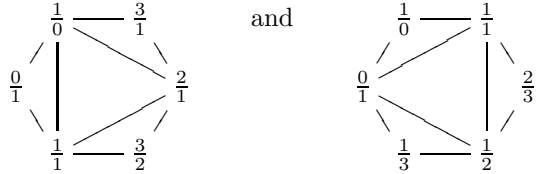
\xymatrix @!0 @R=0.45cm @C=0.55cm
{
&\frac{1}{0}\ar@{-}[ldd]\ar@{-}[rrrdd]\ar@{-}[dddd]\ar@{-}[rr]&& \frac{3}{1}\ar@{-}[rdd]&\\
\\
\frac{0}{1}\ar@{-}[rdd]&&&&  \frac{2}{1}\ar@{-}[ldd]\\
\\
& \frac{1}{1}\ar@{-}[rr]\ar@{-}[rrruu]&& \frac{3}{2}
}
\qquad\hbox{and}\qquad
\xymatrix @!0 @R=0.45cm @C=0.55cm
{
&\frac{1}{0}\ar@{-}[ldd]\ar@{-}[rr]&& \frac{1}{1}\ar@{-}[rdd]\ar@{-}[llldd]\ar@{-}[dddd]&\\
\\
\frac{0}{1}\ar@{-}[rdd]\ar@{-}[rrrdd]&&&&  \frac{2}{3}\ar@{-}[ldd]\\
\\
& \frac{1}{3}\ar@{-}[rr]&& \frac{1}{2}
}
$$
respectively.
One checks that the first Farey hexagon is obtained from the second one
by the action of the matrix
$
\left(
\begin{array}{cc}
3&-1\\[4pt]
1&0
\end{array}
\right).
$
\end{ex}

\subsection{$3d$-dissections and walks on the Farey tessellation}\label{ConstrFEmSec}

Construction~\eqref{CycEqBis} can be applied with an arbitrary positive solution
of the equation $M(c_1,\ldots,c_N)=\pm\Id$.
This leads again to an $N$-cycle starting at~$\frac{1}{0}$ and ending at~$\frac{0}{1}$. 

In the sequence of vertices defining the cycle a vertex may appear several times 
and it will be important to distinguish~$\frac{r}{s}$ and~$\frac{-r}{-s}$.
In other words, we will consider the {\it twofold covering}
of the projective line over~$\Q$.

\begin{defn}
Given a Farey $n$-gon,
$\left(\frac{1}{0},\frac{r_1}{s_1},\ldots,\frac{r_{n-1}}{s_{n-1}},\frac{0}{1}\right)$
and an integer~$N\geq{}n$,
an $N$-periodic (or antiperiodic) sequence
of its vertices,~$\big(\frac{r_{i_{j}}}{s_{i_{j}}}\big)_{j\in\Z}$,
is called

(i) 
a {\it walk} on the $n$-gon if
$\frac{r_{i_{j}}}{s_{i_{j}}}$ and $\frac{r_{i_{j+1}}}{s_{i_{j+1}}}$
are connected by an edge, for all~$j$;

(ii)
a {\it positive walk} if it is a walk and $r_{i_{j}}s_{i_{j+1}}-s_{i_{j}}r_{i_{j+1}}>0$, for all~$j$.
\end{defn}

In other words, fixing an orientation on the Farey $n$-gon, a positive walk has the same orientation.
\begin{ex}
\label{FareyWEx}
a) 
A Farey $n$-gon itself is an $n$-antiperiodic positive walk.
The antiperiodicity is due to the fact that after arrival at $\frac{r_n}{s_n}=\frac{0}{1}$,
one has to continue with $\frac{-1}{0}$ in order to keep 
$r_ns_{n+1}-s_nr_{n+1}=1$.

Let us give simple concrete examples of positive walks.

b)
Consider the Farey quadrilateral 
$\left(\frac{1}{0},\,\frac{1}{1},\,\frac{1}{2},\,\frac{0}{1}\right)$,
the $7$-periodic walk
$
\left(
\frac{1}{0},\,
\frac{1}{1},\,
\frac{0}{1},\,
\frac{-1}{0},\,
\frac{-1}{-1},\,
\frac{-1}{-2},\,
\frac{0}{-1}
\right)
$
 is positive.
 The quadrilateral and its Farey triangulation
 and the walk are represented by the diagrams:
 $$
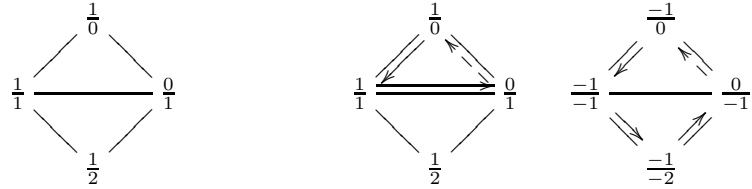
\xyoption{tips}
\xymatrix @!0 @R=1cm @C=1cm
{
&\frac{1}{0}\ar@{-}[rd]\ar@{-}[ld]
\\
\frac{1}{1}\ar@{-}[rr]&&\frac{0}{1}\\
&\frac{1}{2}\ar@{-}[ru]\ar@{-}[lu]
}
\qquad\qquad\qquad
\xymatrix @!0 @R=1cm @C=1cm
{
&\frac{1}{0}\ar@{-}[rd]\ar@{->}@<3pt>@*{[red]}[ld]\ar@{-}[ld]
&&&\frac{-1}{0}\ar@{-}[rd]\ar@{->}@<3pt>@*{[red]}[ld]\ar@{-}[ld]
\\
\frac{1}{1}\ar@{->}@<3pt>@*{[red]}[rr]\ar@{-}[rr]&&
\frac{0}{1}\ar@{-->}@<3pt>@*{[red]}[lu]
&\frac{-1}{-1}\ar@{-}[rr]&&\frac{0}{-1}\ar@{-->}@<3pt>@*{[red]}[lu]\\
&\frac{1}{2}\ar@{-}[ru]\ar@{-}[lu]
&&&\frac{-1}{-2}\ar@{->}@<3pt>@*{[red]}[ru]\ar@{-}[ru]\ar@{<-}@<-3pt>@*{[red]}[lu]
\ar@{-}[lu]
}
$$
where the dashed arrow indicates a change of signs in the sequence so that the next step of the walk is drawn in the copy of the $n$-gon with opposite signs.

The following $10$-antiperiodic walk
$
\left(
\frac{1}{0},\,
\frac{1}{1},\,
\frac{1}{2},\,
\frac{0}{1},\,
\frac{-1}{-1},\,
\frac{-1}{-2},\,
\frac{0}{-1},\,
\frac{1}{0},\,
\frac{1}{1},\,
\frac{0}{1}
\right)
$
along the same quadrilateral  is also positive and is represented by
$$
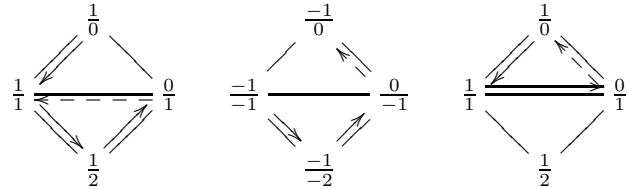
\xymatrix @!0 @R=1cm @C=1cm
{
&\frac{1}{0}\ar@{-}[rd]\ar@{->}@<3pt>@*{[red]}[ld]\ar@{-}[ld]
&&&\frac{-1}{0}\ar@{-}[rd]\ar@{-}[ld]
&&&\frac{1}{0}\ar@{-}[rd]\ar@{->}@<3pt>@*{[red]}[ld]\ar@{-}[ld]
\\
\frac{1}{1}\ar@{-}[rr]\ar@{<--}@<-2pt>@*{[red]}[rr]
&&\frac{0}{1}
&\frac{-1}{-1}\ar@{-}[rr]
&&\frac{0}{-1}\ar@{-->}@<3pt>@*{[red]}[lu]
&\frac{1}{1}\ar@{->}@<3pt>@*{[red]}[rr]\ar@{-}[rr]
&&\frac{0}{1}\ar@{-->}@<3pt>@*{[red]}[lu]
\\
&\frac{1}{2}\ar@{->}@<3pt>@*{[red]}[ru]\ar@{-}[ru]\ar@{<-}@<-3pt>@*{[red]}[lu]\ar@{-}[lu]
&&&\frac{-1}{-2}\ar@{->}@<3pt>@*{[red]}[ru]\ar@{-}[ru]\ar@{<-}@<-3pt>@*{[red]}[lu]\ar@{-}[lu]
&&&\frac{1}{2}\ar@{-}[ru]\ar@{-}[lu]
}
$$

c)
Consider the Farey hexagon
$\left(
\frac{1}{0},\,
\frac{1}{1},\,
\frac{2}{3},\,
\frac{1}{2},\,
\frac{1}{3},\,
\frac{0}{1}
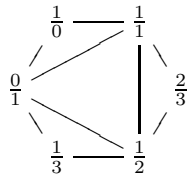
\right)$:
$$
\xymatrix @!0 @R=0.45cm @C=0.55cm
{
&\frac{1}{0}\ar@{-}[ldd]\ar@{-}[rr]
&& \frac{1}{1}\ar@{-}[rdd]\ar@{-}[llldd]\ar@{-}[dddd]&\\
\\
\frac{0}{1}\ar@{-}[rdd]\ar@{-}[rrrdd]
&&&&  \frac{2}{3}\ar@{-}[ldd]\\
\\
& \frac{1}{3}\ar@{-}[rr]&& \frac{1}{2}
}
$$
The $9$-periodic walk
$\left(
\frac{1}{0},\,
\frac{1}{1},\,
\frac{1}{2},\,
\frac{1}{3},\,
\frac{0}{1},\,
\frac{-1}{-1},\,
\frac{-2}{-3},\,
\frac{-1}{-2},\,
\frac{0}{-1}
\right)$
is positive and is represented by the diagram
$$
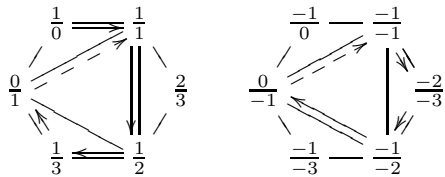
\xymatrix @!0 @R=0.45cm @C=0.55cm
{
&\frac{1}{0}\ar@{-}[ldd]\ar@{->}@<-2pt>@*{[red]}[rr]\ar@{-}[rr]&&
 \frac{1}{1}\ar@{-}[rdd]\ar@{-}[llldd]\ar@{->}@<-3pt>@*{[red]}[dddd]\ar@{-}[dddd]
&&&&\frac{-1}{0}\ar@{-}[ldd]\ar@{-}[rr]
&& \frac{-1}{-1}\ar@{->}@<-3pt>@*{[red]}[rdd]\ar@{-}[llldd]\ar@{-}[dddd]\ar@{-}[rdd]\ar@{-}[rdd]
&
\\
\\
\frac{0}{1}\ar@{<-}@<3pt>@*{[red]}[rdd]\ar@{-}[rrrdd]\ar@{-}[rdd]\ar@{-->}@<-3pt>@*{[red]}[rrruu]
&&&&  \frac{2}{3}\ar@{-}[ldd]
&&\frac{0}{-1}\ar@{-}[rdd]\ar@{<-}@<3pt>@*{[red]}[rrrdd]\ar@{-}[rrrdd]\ar@{-->}@<-3pt>@*{[red]}[rrruu]
&&&&  \frac{-2}{-3}\ar@{->}@<-3pt>@*{[red]}[ldd]\ar@{-}[ldd]
\\
\\
& \frac{1}{3}\ar@{<-}@<2pt>@*{[red]}[rr]\ar@{-}[rr]
&& \frac{1}{2}
&&&& \frac{-1}{-3}\ar@{-}[rr]&& \frac{-1}{-2}
}
$$
\end{ex}

Recall (cf. Section~\ref{ConstrWalkFSec}) that every solution of the equation
$M(c_1,\ldots,c_N)=\pm\Id$ defines an $N$-(anti)periodic positive walk on 
some Farey $n$-gon, where~$n\leq{}N$.
The following theorem is nthe main result of this section.

\begin{thm}
\label{FWalkThm}
(i)
Every $N$-(anti)periodic positive walk on a Farey $n$-gon corresponds to a solution of the equation
$M(c_1,\ldots,c_N)=\pm\Id$.

(ii)
Conversely, every solution of
$M(c_1,\ldots,c_N)=\pm\Id$ 
can be obtained from an
$N$-(anti)periodic walk on
Farey $n$-gons with $n\leq{}N$.
\end{thm}

\begin{proof}
Part (i).
Consider an $N$-(anti)periodic positive walk 
$\left(\frac{r_i}{s_i}\right)_{1\leq{}i\leq{}N}$.
Since for every  $i$ we have
\begin{equation}
\label{WronEq}
\det
\left(
\begin{array}{cc}
r_i&r_{i+1}\\[4pt]
s_i&s_{i+1}
\end{array}
\right)=
r_is_{i+1}-r_{i+1}s_i=1,
\end{equation}
both, the numerator and the denominator must satisfy a linear
recurrence
$$
V_{i+1}=c_iV_i-V_{i-1},
$$
with some $N$-periodic sequence $(c_i)_{i\in\Z}$.
The monodromy of this equation is the matrix
$M(c_1,\ldots,c_N)$.
Since $(r_i)_{i\in\Z}$ and $(s_i)_{i\in\Z}$ are two
(anti)periodic solutions, this monodromy is equal to $\pm\Id$.

Part (ii). 
Given a solution of the equation $M(c_1,\ldots,c_N)=\pm\Id$,
applying the construction~\eqref{CycEq}, 
one obtains a sequence of rationals $\big(\frac{r_i}{s_i}\big)_{i\in\Z}$,
which is periodic or antiperiodic depending on the sign in the right-hand-side of the equation
and satisfies~\eqref{WronEq}.
To prove that this sequence is indeed a walk on a Farey $n$-gon, one needs to show
that every pair of neighbors $\frac{r_i}{s_i}$ and $\frac{r_j}{s_j}$
(i.e., such points that there is no other point in the sequence in the interval $\big(\frac{r_i}{s_i},\frac{r_j}{s_j}\big)$)
is connected by an edge in the Farey graph.
Suppose that $\frac{r_i}{s_i}$ and $\frac{r_j}{s_j}$ are not connected.
Assume that $\frac{r_i}{s_i}<\frac{r_j}{s_j}$.
Then either $\frac{r_{i+1}}{s_{i+1}}>\frac{r_j}{s_j}$, or $\frac{r_{i-1}}{s_{i-1}}>\frac{r_j}{s_j}$,
and both of these points must be connected to $\frac{r_i}{s_i}$.
Similarly, either $\frac{r_{j+1}}{s_{j+1}}<\frac{r_i}{s_j}$, or $\frac{r_{j-1}}{s_{j-1}}<\frac{r_i}{s_i}$,
and both of these points must be connected to $\frac{r_j}{s_j}$.
This means that there are crossing edges (geodesics in the Farey tessellation), which is a contradiction.
\end{proof}

Theorem~\ref{FWalkThm} establishes a one-to-one correspondence between
positive $N$-(anti)periodic walks on the Farey tessellation and solutions of $M(c_1,\ldots,c_N)=\pm\Id$.
Theorem~\ref{FWalkThm} together with Theorem~\ref{SecondMainThm} then imply the following.

\begin{cor}
The sequence $(c_1,\ldots,c_N)$ corresponding to an $N$-(anti)periodic positive walk
is a quiddity of a $3d$-dissection of an $N$-gon.
\end{cor}

\begin{ex}
Let us continue Example~\ref{FareyWEx}.

The $7$-periodic walk 
$
\left(
\frac{1}{0},\,
\frac{1}{1},\,
\frac{0}{1},\,
\frac{-1}{0},\,
\frac{-1}{-1},\,
\frac{-1}{-2},\,
\frac{0}{-1}
\right)
$
generates the $7$-periodic sequence $(c_i)_{i\in\Z}$ with the period
$(1,1,1,1,2,1,2)$.
The $10$-antiperiodic walk
$
\left(
\frac{1}{0},\,
\frac{1}{1},\,
\frac{1}{2},\,
\frac{0}{1},\,
\frac{-1}{-1},\,
\frac{-1}{-2},\,
\frac{0}{-1},\,
\frac{1}{0},\,
\frac{1}{1},\,
\frac{0}{1}
\right)
$
generates the $10$-periodic sequence $(c_i)_{i\in\Z}$ with the period
$(1,2,1,1,1,1,2,1,1,1)$.
These are quiddities of the $3d$-dissections
$$
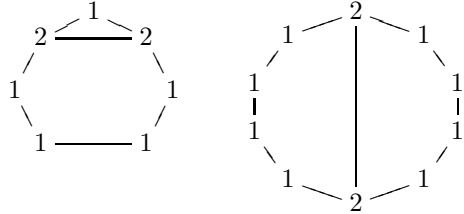
\xymatrix @!0 @R=0.35cm @C=0.35cm
{
&&&1\ar@{-}[rrd]\ar@{-}[lld]
\\
&2\ar@{-}[ldd]\ar@{-}[rrrr]&&&& 2\ar@{-}[rdd]&\\
\\
1\ar@{-}[rdd]&&&&&& 1\ar@{-}[ldd]\\
\\
&1\ar@{-}[rrrr]&&&&1
}
\qquad
\xymatrix @!0 @R=0.32cm @C=0.45cm
 {
&&&2\ar@{-}[dddddddd]\ar@{-}[lld]\ar@{-}[rrd]&
\\
&1\ar@{-}[ldd]&&&& 1\ar@{-}[rdd]\\
\\
1\ar@{-}[dd]&&&&&&1\ar@{-}[dd]\\
\\
1\ar@{-}[rdd]&&&&&&1\ar@{-}[ldd]\\
\\
&1\ar@{-}[rrd]&&&& 1\ar@{-}[lld]\\
&&&2&
}
$$
respectively.

The $9$-periodic walk
$\left(
\frac{1}{0},\,
\frac{1}{1},\,
\frac{1}{2},\,
\frac{1}{3},\,
\frac{0}{1},\,
\frac{-1}{-1},\,
\frac{-2}{-3},\,
\frac{-1}{-2},\,
\frac{0}{-1}\right)$
generates the quiddity of the following 
$3d$-dissection of a nonagon:
$$
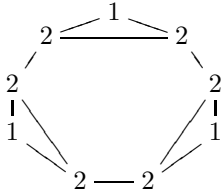
\xymatrix @!0 @R=0.32cm @C=0.45cm
 {
&&&1\ar@{-}[lld]\ar@{-}[rrd]&
\\
&2\ar@{-}[ldd]\ar@{-}[rrrr]&&&& 2\ar@{-}[rdd]\\
\\
2\ar@{-}[dd]\ar@{-}[rrdddd]&&&&&&2\ar@{-}[dd]\ar@{-}[lldddd]\\
\\
1\ar@{-}[rrdd]&&&&&&1\ar@{-}[lldd]\\
\\
&&2\ar@{-}[rr]&& 2&
}
$$
\end{ex}

We can say that while walking on a triangulated $n$-gon, the ``invisible hand'' draws
a $3d$-dissection of an $N$-gon with $N>n$.

\subsection{The quiddity of a $3d$-dissection from a Farey walk}\label{RecQWSec}

Given an $n$-gon,
$\left(\frac{1}{0},\frac{r_1}{s_1},\ldots,\frac{r_{n-1}}{s_{n-1}},\frac{0}{1}\right)$,
and an $N$-(anti)periodic walk on it~$\big(\frac{r_{i_{j}}}{s_{i_{j}}}\big)_{j\in\Z}$,
it is natural to ask, how to recover the sequence $(c_1,\ldots,c_N)$
which is a solution of the equation $M(c_1,\ldots,c_N)=\pm\Id$.
The answer is as follows.

The integer $c_{i_j}$ counts the number of triangles in the $n$-gon that lie
on the positive side of the walk,
$$
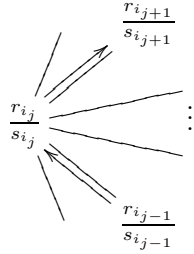
\xymatrix @!0 @R=0.45cm @C=0.55cm
{
&&&
\frac{r_{i_{j+1}}}{s_{i_{j+1}}}\ar@{-}[lllddd]\ar@{<-}@<-3pt>@*{[red]}[lllddd]
&\\
&&&&&&
\\
&&&&&&\\
\frac{r_{i_{j}}}{s_{i_{j}}}\ar@{<-}@<-3pt>@*{[red]}[rrrddd]\ar@{-}[rrrddd]\ar@{-}[rddd]\ar@{-}[rrrrd]^{}
\ar@{-}[rrrru]\ar@{-}[ruuu]
&&&&  \vdots
\\
&&&&&&
\\&&&&&&
\\
&
&&\frac{r_{i_{j-1}}}{s_{i_{j-1}}}
&
}
$$
with respect to the orientation of the hyperbolic plane.
This is a quiddity of a $3d$-dissection of an $N$-gon, as follows from
Theorem~\ref{SecondMainThm}.

\section{PPP: Ptolemy,  Pl\"ucker and Pfaff}\label{PPPSec}

In this section we prove that every solution
$(c_1,\ldots,c_n)$ of the equation $M(c_1,\ldots,c_n)=\pm\Id$, with $c_i$ positive integers,
defines a certain labeling of the diagonals of a convex $n$-gon:
$$
x:V\times{}V\to\Z,
$$
where $V$ is the set of vertices of the $n$-gon,
usually identified with $\{1,\ldots,n\}$.
Moreover, the set of integers~$x_{i,j}$,
satisfies the Ptolemy-Pl\"ucker relations.
In this sense, the results discussed in Section~\ref{CFPtoRuSec}
are still valid in the case where no continued fraction is defined.
Note that the totally positive solutions are in a one-to-one correspondence
with the labelings where all~$x_{i,j}$ are positive integers.
For an arbitrary solution,
one can only guarantee that the shortest diagonals are labeled by positive integers.

This section contains the proofs of the main results.
Our main tool is the well-known polynomial called (Euler's) continuant.
This is the determinant of a tridiagonal matrix, it gives an explicit formula for the entries
of the matrices $M(c_1,\ldots,c_n)$.
The Ptolemy-Pl\"ucker relations are deduced from
the Euler identity for the continuants.

We conclude the section with the similar ``Pfaffian formulas'' for the trace
$\tr\!\!(M(c_1,\ldots,c_n))$ recently obtained in~\cite{CoOv}.
The proof is more technical and we do not dwell on it.

\subsection{Continuant $=$ ``continued fraction determinant''}\label{ContSec}
The material of this subsection is classical.

Let us think of $(c_1,\ldots,c_n)$ as formal commuting variables, and
consider the negative continued fraction:
\begin{equation}
\label{FinCFEq}
\frac{r_n}{s_n}
=\llbracket{}c_1,\ldots,c_n\rrbracket,
\end{equation}
here and below~$n\geq1$.
Then both the numerator and the denominator are certain polynomials in $c_i$.
It turns out that these polynomials are basically the same.

\begin{defn}
The tridiagonal determinant 
$$
K_{n}(c_1,\ldots,c_n):=
\det\left(
\begin{array}{cccccc}
c_1&1&&&\\[4pt]
1&c_{2}&1&&\\[4pt]
&\ddots&\ddots&\!\!\ddots&\\[4pt]
&&1&c_{n-1}&\!\!\!\!\!1\\[4pt]
&&&\!\!\!\!\!1&\!\!\!\!c_{n}
\end{array}
\right)
$$
is called the {\it continuant}.
We also set for convenience~$K_0:=1$ and~$K_{-1}:=0$.
\end{defn}

The following statement is commonly known.
The proof is elementary and we give it for the sake of completeness.

\begin{prop}
\label{EulProp}
The numerator and the denominator of~\eqref{FinCFEq} are given by the continuants
\begin{equation}
\label{ContCFEq}
\left\{
\begin{array}{rcl}
r_n&=&K_{n}(c_1,\ldots,c_n),\\[4pt]
s_n&=&K_{n-1}(c_2,\ldots,c_n).
\end{array}
\right.
\end{equation}
\end{prop}

\begin{proof}
Formula~(\ref{ContCFEq}) follows from the recurrence relation
\begin{equation}
\label{DEqEq}
V_{i+1}-c_{i+1}V_{i}+V_{i-1}=0,
\end{equation}
with (known) coefficients~$(c_i)_{i\in\Z}$ and (indeterminate) sequence~$(V_i)_{i\in\Z}$.
Let $\frac{r_i}{s_i}=\llbracket{}c_1,\ldots,c_i\rrbracket$ be a convergent 
of the continued fraction~\eqref{FinCFEq}, then
both sequences,~$(r_i)_{i\geq1}$ and~$(s_i)_{i\geq1}$ satisfy~(\ref{DEqEq}),
with the initial conditions
$(r_1,r_2)=(c_1,\,c_1c_2-1)$ and $(s_1,s_2)=(1,\,c_2)$.
Indeed, this is equivalent to Proposition~\ref{MatConv}.
On the other hand, the continuants satisfy
\begin{equation}
\label{DEqContEq}
K_{i}(c_1,\ldots,c_i)=c_iK_{i-1}(c_1,\ldots,c_{i-1})-K_{i-2}(c_1,\ldots,c_{i-2}).
\end{equation}
Hence the result.
\end{proof}

As a consequence of~(\ref{ContCFEq}), we obtain the following
formula for the entries of the matrix $M(c_1,\ldots,c_n)$:
\begin{equation}
\label{SLBisEq}
M(c_1,\ldots,c_n)
=\left(
\begin{array}{cc}
K_{n}(c_1,\ldots,c_n)&-K_{n-1}(c_1,\ldots,c_{n-1})\\[8pt]
K_{n-1}(c_2,\ldots,c_n)&-K_{n-2}(c_2,\ldots,c_{n-1})
\end{array}
\right).
\end{equation}
Indeed, $M(c_1,\ldots,c_n)$ is the matrix of convergents, cf. Section~\ref{MatConverSec}.

\subsection{The Euler identity for continuants}\label{ConEyltSec}

The polynomials $K_{n}(c_1,\ldots,c_n)$ were studied by Euler who proved the following identity
(see, e.g.,~\cite{Concr}).

\begin{thm}[Euler]
For $1\leq{}i\leq{}j<k\leq{}\ell\leq{}n$, one has
\label{EulerThm}
\begin{equation}
\begin{array}{l}
\label{EulIdEq}
K_{k-i}(c_i,\ldots,c_{k-1})K_{\ell-j}(c_{j+1},\ldots,c_{\ell})=\\[4pt]
\qquad
K_{j-i}(c_i,\ldots,c_{j-1})K_{\ell-i}(c_{i+1},\ldots,c_{\ell})+
K_{\ell-i+1}(c_i,\ldots,c_{\ell})K_{k-j-1}(c_{j+1},\ldots,c_{k-1}).
\end{array}
\end{equation}
\end{thm}

We give here an elegant proof due to A. Ustinov~\cite{Ust}
that makes use of the Pfaffian of a skew-symmetric matrix.

\begin{proof}
Using the notation
$$
x_{i-1,j+1}:=K_{j-i+1}(c_i,\ldots,c_j),
$$
for $i<j$,
consider the $4\times4$ skew-symmetric matrix
$$
\Omega=
\left(
\begin{array}{cccc}
0&x_{i+1,j}&x_{i+1,k}&x_{i+1,\ell}\\[4pt]
-x_{i+1,j}&0&x_{j,k}&x_{j,\ell}\\[4pt]
-x_{i+1,k}&-x_{j,k}&0&x_{k,\ell}\\[4pt]
-x_{i+1,\ell}&-x_{j,\ell}&-x_{k,\ell}&0
\end{array}
\right).
$$
It readily follows from the recurrence relation~\eqref{DEqContEq}, that the matrix
$\Omega$ has rank~$2$.
Hence
$$
\det(\Omega)=\left(
x_{i+1,j}x_{k,\ell}+x_{i+1,\ell}x_{j,k}-x_{i+1,k}x_{j,\ell}
\right)^2=0,
$$
which is precisely~\eqref{EulIdEq}.
\end{proof}

\begin{rem}
a)
Formula~(\ref{ContCFEq}) allows one to work with continued fractions with
$c_i$ assigned to concrete numbers (integers, real, complex, etc.), even when
the ``naive'' expression~(\ref{FinCFEq}) is not well-defined.
This may happen when some denominators vanish, for instance if several consecutive coefficients
$c_i,c_{i+1},\ldots$ are equal to $1$.

b)
Replacing the negative continued fraction by a regular 
one:~$\frac{r_n}{s_n}=[a_1,\ldots,a_n]$, where~$n$ can be even or odd,
formula~(\ref{ContCFEq}) is replaced by a similar formula with the only difference that
the continuant is replaced by the determinant
$$
K^+_{n}(a_1,\ldots,a_n):=
\det\left(
\begin{array}{cccccc}
a_1&1&&&\\[4pt]
\!-1&a_{2}&1&&\\[4pt]
&\ddots&\ddots&\!\!\ddots&\\[4pt]
&&-1&a_{n-1}&\!\!\!\!\!1\\[4pt]
&&&\!\!\!\!-1&\!\!\!\!a_{n}
\end{array}
\right),
$$
also known under the name of continuant.

c)
The continuants enjoy many remarkable properties (some of which are listed in~\cite{Concr,BR1,CoOv}).
They were already known to Euler
who thoroughly studied the polynomials $K_n$
and, in particular, established Ptolemy-type identities for them.
In a sense, the continuants establish a relationship between the continued fractions and projective geometry;
see~\cite{SVRS} and references therein.

d)
Let us mention that equation~(\ref{DEqEq}) is called the discrete Sturm-Liouville, Hill, or Schr\"odinger equation.
It plays an important role in many areas of algebra, analysis and mathematical physics.
When the sequence of coefficients is periodic:
$c_{i+n}=c_{i}$, for all $i$, there is a notion of monodromy matrix of~(\ref{DEqEq}),
which is nothing else than the matrix~$M(c_1,\ldots,c_n)$.
\end{rem}

\subsection{Ptolemy-Pl\"ucker relations}\label{PtoSec}
We are ready to explain the connection between solutions of
the equation $M(c_1,\ldots,c_n)=\pm\Id$
and the Ptolemy-Pl\"ucker relations.

Let $x_{i,j}$, where $i,j\in\{1,\ldots,n\}$,
be a set of $n^2$ formal (commuting) variables.
In order to have a clear combinatorial picture, and following~\cite{FZ1},
we will always think of an $n$-gon with the vertices cyclically ordered by $\{1,\ldots,n\}$,
and the diagonals $(i,j)$ labeled by $x_{i,j}$.

We call the {\it Ptolemy-Pl\"ucker relations} the following system of equations
\begin{equation}
\label{PPCompleteEqBis}
\left\{
\begin{array}{rcl}
x_{i,j}\,x_{k,\ell}&=&x_{i,k}\,x_{j,\ell}+x_{i,\ell}\,x_{k,j},
\qquad
i\leq{}k\leq{}j\leq{}\ell,\\[4pt]
x_{i,i}&=&0,
\\[4pt]
x_{i,i+1}&=&1.
\end{array}
\right.
\end{equation}
We will distinguish two special cases, where the set of variables~$x_{i,j}$ is either symmetric,
or skew-symmetric:
$$
x_{i,j}=x_{j,i},
\qquad\hbox{or}\qquad
x_{i,j}=-x_{j,i}.
$$

The following statement arose as an attempt to interpret some of the results
of~\cite{Cox} (see also~\cite{SVRS}).

\begin{thm}
\label{PPCompleteThm}
Let $(c_1,\ldots,c_n)$ be positive integers. The system \eqref{PPCompleteEqBis} together with the symmetry condition $x_{i,j}=x_{j,i}$ has a unique solution such that  $x_{i-1,i+1}=c_{i}$ 
 if and only if
one has $M(c_1,\ldots,c_n)=-\Id$.
\end{thm}

\begin{proof}
We use the following lemma.

\begin{lem}
\label{CoxLem}
If $(x_{i,j})$ satisfies 
 the Ptolemy-Pl\"ucker relations \eqref{PPCompleteEqBis} then for all $i\leq{}j$
one has
\begin{equation}
\label{LabEq}
x_{i-1,j+1}=
\det\left(
\begin{array}{cccccc}
c_i&1&&&\\[4pt]
1&c_{i+1}&1&&\\[4pt]
&\ddots&\ddots&\!\!\ddots&\\[4pt]
&&1&c_{j-1}&\!\!\!\!\!1\\[4pt]
&&&\!\!\!\!\!1&\!\!\!\!c_{j}
\end{array}
\right)=K_{j-i+1}(c_i,\ldots,c_j).
\end{equation}
where $c_{i}=x_{i-1,i+1}$. 
\end{lem}

\begin{proof}
We proceed by induction on $j$.

The induction base is~$x_{i-1,i+1}=c_i=K_1(c_i)$ by definition, and the following calculation of~$x_{i-1,i+2}$.
$$
 
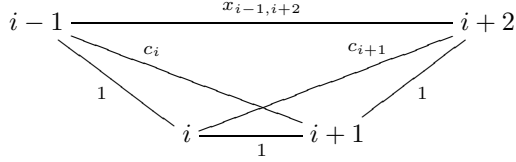
\xymatrix@!0 @R=1.5cm @C=2cm
 {
i-1\ar@{-}[rrd]^<<<<<<<<<<<{c_{i}}\ar@{-}[rd]_{1}&&&
i+2\ar@{-}[lld]_<<<<<<<<<<<{c_{i+1}}\ar@{-}[ld]^{1}\ar@{-}[lll]_{x_{i-1,i+2}}\\
&i&i+1\ar@{-}[l]^{1}
}
$$
The Ptolemy-Pl\"ucker relation reads
$c_{i}c_{i+1}=x_{i-1,i+2}+1$,
hence $x_{i-1,i+2}=K_2(c_{i},c_{i+1})$.

The induction step consists of expanding the determinant of  \eqref{LabEq} with respect to the last column and compare with the Ptolemy-Pl\"ucker relation given by the diagram

$$
 
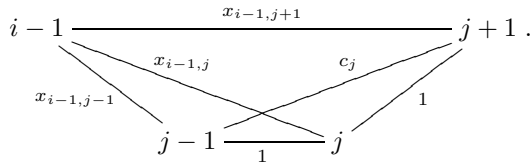
\xymatrix@!0 @R=1.5cm @C=2cm
 {
i-1\ar@{-}[rrd]^<<<<<<<<<<<<<<{x_{i-1,j}}\ar@{-}[rd]_{x_{i-1,j-1}}&&&
j+1\ar@{-}[lld]_<<<<<<<<<<<<<<{c_{j}}\ar@{-}[ld]^{1}\ar@{-}[lll]_{x_{i-1,j+1}}\\
&j-1&j\ar@{-}[l]^{1}
}.
$$
Both relations are equivalent to~\eqref{DEqContEq}.
Hence the lemma.
\end{proof}

Let us show that the Ptolemy-Pl\"ucker relations imply $M(c_1,\ldots,c_n)=-\Id$.
Applying Lemma~\ref{CoxLem} to the case $\vert j-i\vert=n-1$, and using the cyclic numeration
of the vertices of the $n$-gon, we get
$$
x_{i,i-1}=K_{n}(x_i,\ldots,x_{i+n-1})=1,
$$
provided $x_{i,j}=x_{j,i}$.
Then, again implying~\eqref{PPCompleteEqBis}, one readily gets 
$$
K_{n+1}(x_i,\ldots,x_{i+n})=0,
\qquad
K_{n+2}(x_i,\ldots,x_{i+n+1})=-1.
$$
We conclude by~\eqref{SLBisEq}, that $M(c_1,\ldots,c_n)=-\Id$.

Conversely, assume that $M(c_1,\ldots,c_n)=-\Id$.
Starting from $x_{i,i}=0$, $x_{i,i+1}=1$ and then labeling
the diagonals of the $n$-gon using~\eqref{LabEq} one obtains a solution
of  \eqref{PPCompleteEqBis} using the Euler identities \eqref{EulerThm} for the continuants.
\end{proof}

A similar computation (that we omit) allows one to prove the following skew-symmetric
counterpart of Theorem~\ref{PPCompleteThm}.

\begin{thm}
\label{PPCompleteThmBis}
Let $(c_1,\ldots,c_n)$ be positive integers. The system \eqref{PPCompleteEqBis} together with the skew-symmetry condition $x_{i,j}=-x_{j,i}$ has a unique solution such that  $x_{i-1,i+1}=c_{i}$ 
 if and only if
$M(c_1,\ldots,c_n)=\Id$.
\end{thm}

\begin{rem}
Let us mention that the coordinates $x_{i,j}$ satisfying~(\ref{PPCompleteEqBis}) together with the
skew-symmetry condition can be identified with
the Pl\"ucker coordinates of the Grassmannian~$G_{2,n}$ of $2$-dimensional subspaces
in the $n$-dimensional vector space.
The coordinate ring of~$G_{2,n}$ is one of the basic examples
of cluster algebras of Fomin and Zelevinsky~\cite{FZ} (for details; see~\cite{FZ1}).
A description of the relationship between Coxeter's frieze patterns and cluster algebras can be found
in~\cite{Sop}.
\end{rem}

\subsection{Relation to $3d$-dissections}\label{Back3dSec}
So far in this section we considered formal variables $x_{i,j}$.
Assigning concrete integral values to these variables, one has to deal with integer solutions of the equation
$M(c_1,\ldots,c_n)=\pm\Id$.
In particular, Theorems~\ref{SecondMainThm},~\ref{PPCompleteThm} 
and~\ref{PPCompleteThmBis} 
 imply the following relation to $3d$-dissections (see Section~\ref{MoySec}).

\begin{cor}
Given an $n$-tuple of positive integers $(c_1,\ldots,c_n)$, start labeling the diagonals of an $n$-gon
by 
$$
x_{i,i}=0,
\qquad 
x_{i,i+1}=1,
\qquad
x_{i-1,i+1}=c_i,
$$
for all $1\leq{}i\leq{}n$,
and then continue using the Ptolemy-Pl\"ucker relations.
This procedure is consistent, and there exists a set of integers $x_{i,j}$ 
and satisfying~(\ref{PPCompleteEqBis}), if and only if $(c_1,\ldots,c_n)$ is a quiddity of a $3d$-dissection.
\end{cor}

\subsection{Traces and Pfaffians}\label{TPfSec}

Let us give one more determinant formula.
We are interested in calculating the trace of the matrix $M(c_1,\ldots,c_n)$.
It follows from~(\ref{SLBisEq}) that this trace is equal to the difference of two
continuants:
$$
\tr\!\!(M(c_1,\ldots,c_n))=
K_{n}(c_1,\ldots,c_n)-K_{n-2}(c_2,\ldots,c_{n-1}).
$$

It turns out that the square of this polynomial is equal to the determinant of
a $2n\times2n$ matrix.

\begin{thm}[\cite{CoOv,CoOv1}]
\label{CoOvThm}
The trace of the matrix~$M(c_1,\ldots,c_n)$ is equal to the square root of the
determinant of the following skew-symmetric $2n\times2n$ matrix
\begin{equation}
\label{TheOmEq}
\det\left(
\begin{array}{cccc|cccccc}
&&&\;\;1&c_1&1&\\[2pt]
&&&&1&c_2&1&\\
&&&&&\ddots&\ddots&\ddots\\
&&&&&&\ddots&\ddots&1\\
-1&&&&&&&1&c_{n}\\[2pt]
\hline
-c_1&-1&&&&&&&1\\
-1&\ddots&\!\!\!\!\!\!\ddots&&&&&&\\
&\ddots&\ddots&\\
&&&\;\;-1&&&&\\
&&-1&-c_{n}&\!\!-1&&
\end{array}
\right)
\quad=\quad
(\tr\!M(c_1,\ldots,c_n))^2.
\end{equation}
\end{thm}

In other words, $\tr\!\!(M(c_1,\ldots,c_n))$ is the Pfaffian of the matrix on the left-hand-side of~(\ref{TheOmEq}).
We refer to~\cite{CoOv} for a proof of this result.
Let us mention that
formula~(\ref{TheOmEq}) reflects a relation to symplectic geometry; see~\cite{CoOv1}.
More precisely, the $2n\times2n$ matrix in~(\ref{TheOmEq}) appears as the Gram matrix of
the symplectic form in the standard symplectic space evaluated on a Lagrangian configuration.
This relation deserves further investigation.

\begin{ex}
In the case $n=3$ one can easily check directly that
$$
\det\left(
\begin{array}{cccccc}
0&0&1&c_1&1&0\\[4pt]
0&0&0&1&c_2&1\\[4pt]
-1&0&0&0&1&c_3\\[4pt]
-c_1&-1&0&0&0&1\\[4pt]
-1&-c_2&-1&0&0&0\\[4pt]
0&-1&-c_3&-1&0&0
\end{array}
\right)
=\left(c_1c_2c_3-c_1-c_2-c_3\right)^2,
$$
which is nothing other than the square of the trace of $M(c_1,c_2,c_3)$.
\end{ex}

\begin{rem}
If one wants to check~\eqref{TheOmEq} with the computer
and forgets to put the minus sign,
here is what one will obtain for~$n\geq2$:
\begin{equation}
\label{TrRem}
\det\left(
\begin{array}{cccc|cccccc}
&&&\;\;1&c_1&1&\\[2pt]
&&&&1&c_2&1&\\
&&&&&\ddots&\ddots&\ddots\\
&&&&&&\ddots&\ddots&1\\
1&&&&&&&1&c_{n}\\[2pt]
\hline
c_1&1&&&&&&&1\\
1&\ddots&\!\!\!\!\!\!\ddots&&&&&&\\
&\ddots&\ddots&\\
&&&\;\;1&&&&\\[2pt]
&&1&c_{n}&\!\!1&&
\end{array}
\right) 
\quad=\quad 
(-1)^n\left((\tr\!M(c_1,\ldots,c_n))^2-4\right).
\end{equation}
Note that the expression in the right-hand-side of~\eqref{TrRem}
is  the discriminant of the characteristic polynomial of~$M(c_1,\ldots,c_n)$.
It will appear again in Section~\ref{ParamCCSec}.
\end{rem}

\section{Minimal presentation of $\PSL(2,\Z)$}\label{DecSec}

The group $\SL(2,\Z)$ (and thus $\PSL(2,\Z)$) is generated by two elements, and a standard 
choice of generators is either $\{S,R\}$, $\{S,U\}$, or $\{U,R\}$, where 
$$
S=\left(
\begin{array}{cc}
0&-1\\[4pt]
1&0
\end{array}
\right),
\qquad
R=\left(
\begin{array}{cc}
1&1\\[4pt]
0&1
\end{array}
\right),
\qquad
U=\left(
\begin{array}{cc}
1&-1\\[4pt]
1&0
\end{array}
\right).
$$
A natural question is how to make such a presentation canonical.

It is a simple and well-known fact that
every element $A\in\PSL(2,\Z)$
can be presented in the form $A=M(c_1,\ldots,c_k)$
where~$c_i$ are positive integers.
The above question is equivalent to the existence of a canonical presentation in this form.
This question was considered and answered (modulo conjugation of~$A$) in~\cite{HZ,Zag,Kat}.
The coefficients $(c_1,\ldots,c_k)$ are obtained as the (minimal) period of the negative
continued fraction of a fixed point of~$A$.
This fixed point is a quadratic irrationality.

We show the existence and uniqueness of the ``minimal presentation'', $A=M(c_1,\ldots,c_k)$, with $c_i\geq1$, 
and $k$ is the smallest possible.
The coefficients $(c_1,\ldots,c_k)$ of this presentation are calculated via expansion of a rational number
(the quotient of largest coefficients of~$A$).
This statement looks quite surprising since it recovers the period of a quadratic irrationality 
from a continued fraction of a rational.

\subsection{Parametrizing the conjugacy classes in $\PSL(2,\Z)$}\label{ParamCCSec}
Let us outline the history of the problem discussed in this section.
The matrices $M(c_1,\ldots,c_k)$,  with $c_i\geq2$ for every~$i$,
were used to parametrize conjugacy classes of
{\it hyperbolic elements} of~$\PSL(2,\Z)$ 
(recall that $A\in\PSL(2,\Z)$ is hyperbolic if~$\left|\tr\!A\right|\geq3$).

Consider the real projective line~$\RP^1$,
which is identified with~$\R\cup\{\infty\}$ by choosing an affine co\-ordinate~$x$.
The action of~$\PSL(2,\Z)$ on~$\RP^1$ is given by linear-fractional transformations, viz
$$
A=\begin{pmatrix}
a&b\\[2pt]
c&d
\end{pmatrix}:x
\,\to\;
\frac{ax+b}{cx+d}\,.
$$
When~$A$ is hyperbolic, it has two fixed points $x_\pm\in\RP^1$:
$$
x_\pm=\frac{a-d\pm\sqrt{(a+d)^2-4}}{2c}.
$$
Note that the expression $(\tr\!{}A)^2-4$ appeared in~\eqref{TrRem}.
When choosing the representative $A\in\PSL(2,\Z)$ with~$\tr\!A>0$,
the point~$x_+$ has the property that,
for all~$x\not=x_-$, $A^m(x)$ tends to~$x_+$, when~$m\to\infty$.
The point~$x_+$ is thus called the {\it attractive} fixed point of~$A$.

Since $x_\pm$ are quadratic irrationals,
the corresponding continued fractions are periodic (starting from some place) by Lagrange's theorem; see, e.g.,~\cite{Poo,Never}.
Consider the negative continued fraction expansion of 
the attractive fixed point:
$$
x_+=\left\llbracket{}c_1,\ldots,c_\ell,\,\overline{c_{\ell+1},\ldots,c_{\ell+k}}\right\rrbracket,
$$
where $(c_{\ell+1},\ldots,c_{\ell+k})$ is the minimal period of the continued fraction.

The statement explained in~\cite[pp.90--92]{Zag} can be formulated as follows.

\begin{prop}
\label{WNProp}
Every hyperbolic element $A\in\PSL(2,\Z)$ is conjugate to
$M(c_{\ell+1},\ldots,c_{\ell+k})$, and the $k$-tuple $(c_{\ell+1},\ldots,c_{\ell+k})$,
defined modulo cyclic permutations,
characterises the conjugacy class of~$A$ uniquely.
\end{prop}

We refer to~\cite{HZ,Zag} and~\cite{Kat} for a detailed and very clear treatment of
this statement and its applications.

\begin{ex}
\label{CZKEx}
Consider the matrix
$
A=\begin{pmatrix}
10&3\\[2pt]
3&1
\end{pmatrix},
$
whose attractive fixed point is
$x_+=\frac{3+\sqrt{13}}{2}$.
It's continued fraction expansion reads
$x_+=
[3,\overline{3}]=
\left\llbracket{}4,\,\overline{2,\,2,\,5}\right\rrbracket.
$
According to Proposition~\ref{WNProp}, the matrix~$A$ must be conjugate to
$$
M(2,\,2,\,5)=\begin{pmatrix}
13&-3\\[2pt]
9&-2
\end{pmatrix},
$$
and, indeed, one checks that
$A=M(4)\,M(2,\,2,\,5)\,M(4)^{-1}.$
\end{ex}

Proposition~\ref{WNProp} and its impact for number theory,
see, e.g.,~\cite{DIT} and references therein,
is the main motivation for us to study minimal presentations of elements of~$\PSL(2,\Z)$.
When representing a matrix, the condition~$c_i\geq2$, for all $i$,
cannot always be satisfied
(many interesting matrices need $1$'s at the ends of their minimal presentations).

\subsection{Minimal presentation}\label{DecSSec}
Every element $A\in\PSL(2,\Z)$ can be written
in the form $A=M(c_1,\ldots,c_k)$, where $c_i\geq1$.
We are interested in the shortest presentations of this form.
It turns out that the coefficients~$c_i$ can be recovered from the coefficients
of~$A$, without expansions of quadratic irrationals.

\begin{thm}
\label{SLDecThm}
(i)
The presentation $A=M(c_1,\ldots,c_k)$
with positive integer coefficients~$c_i$ is unique, provided~$k$ is the smallest possible.

(ii)
If $A=\begin{pmatrix}
a&-b\\[2pt]
c&-d
\end{pmatrix}$
where $a,b,c,d>0$ and $a>b$, then 
the coefficients $(c_1,\ldots,c_k)$  
are those of the continued fraction $\frac{a}{c}=\llbracket{}c_1,\ldots,c_k\rrbracket$.
\end{thm}

We also have the following ``minimality criterion''.

\begin{prop}
\label{MinCrProp}
If an element $A$ of $\PSL(2,\Z)$ is written in the form
$A=M(c_1,\ldots,c_k)$, then this is the minimal presentation of $A$,
if and only if $c_i\geq2$, except perhaps for
the ends of the sequence, i.e., for $c_1$, or $c_1,c_2$ and $c_k$, or $c_{k-1},c_k$.
\end{prop}

Before giving the proof of Theorem~\ref{SLDecThm} and Proposition~\ref{MinCrProp},
let us consider several examples and corollaries.

Part~(ii) of Theorem~\ref{SLDecThm} covers all different cases of elements of~$\PSL(2,\Z)$,
modulo multiplication by~$R$ and~$S$ from the right and from the left.
For instance, the following statement treats the case of all matrices with positive coefficients.

\begin{cor}
\label{PosDeCor}
Let $A=\begin{pmatrix}
a&b\\[2pt]
c&d
\end{pmatrix}$
be an element of~$\PSL(2,\Z)$ with $a,b,c,d>0$, one has the following two cases:

(i) if $a>b$, then the minimal presentation of~$A$ is
\begin{equation}
\label{CoeffciEq}
A=M(c_1,\ldots,c_k,2,1,1),
\end{equation}
where $\frac{a}{c}=\llbracket{}c_1,\ldots,c_k\rrbracket$;
and the conjugacy class of~$A$ is parametrized by
$(c_1+1,c_2,\ldots,c_k)$;

(ii) if $a<b$, then the minimal presentation of~$A$ is
\begin{equation}
\label{CoeffciBisEq}
A=M(c_1,\ldots,c_{k-1},\,c_k+1,\,1,1)
\end{equation}
where $\frac{b}{d}=\llbracket{}c_1,\ldots,c_k\rrbracket$;
and the conjugacy class of~$A$ is parametrized by
$(c_1+c_k,c_2,\ldots,c_{k-1})$.
\end{cor}

\begin{proof}
Part~(i).
After multiplication from the right by $R^{-1}$,
the matrix $AR^{-1}$ satisfies the conditions of Part~(ii) of Theorem~\ref{SLDecThm}.
One then uses that $R^{-1}=M(1,1,2,1)$ and $M(2,1,2,1)=\id$ (up to a sign, i.e., in~$\PSL(2,\Z)$).
Hence~\eqref{CoeffciEq}.
Next, one has
$$
RAR^{-1}=M(2,1,1,c_1,\ldots,c_k)=M(c_1+1,c_2,\ldots,c_k).
$$

Part~(ii).
The matrix $A$ becomes as in Part~(ii) of Theorem~\ref{SLDecThm},
when multiplied from the right by $S=M(1,1,2,1,1)$.
Next, since~$R^a=M(a+1,1,1)$, one has
$$
R^{c_k}AR^{-c_k}=M(c_k+1,1,1,c_1,\ldots,c_{k-1})=M(c_1+c_k,c_2,\ldots,c_{k-1}).
$$
Hence the result.
\end{proof}

\begin{rem}
Comparing~\eqref{CoeffciEq} and~\eqref{CoeffciBisEq} to somewhat similar known formulas
in terms of the positive continued fractions (see~\cite{Kar}, Theorem~7.14),
we observe that they are quite different.
Indeed, formulas~\eqref{CoeffciEq} and~\eqref{CoeffciBisEq} use the ``dominant'' (largest) coefficients
of~$A$, while the formulas in~\cite{Kar} use the smallest coefficients.
\end{rem}

Rewriting~\eqref{CoeffciEq} and~\eqref{CoeffciBisEq}
in terms of the standard generators, and using Proposition~\ref{MaTCompProp},
we have the following decomposition.

\begin{cor}
\label{PosDeGenCor}
Let $A=\begin{pmatrix}
a&b\\[2pt]
c&d
\end{pmatrix}$
be an element of~$\PSL(2,\Z)$ with $a,b,c,d>0$. Its expression in terms of the standard generators is:

(i) if $a>b$, then
$$
\begin{array}{rcl}
A&=&R^{c_1}S\,R^{c_{2}}S\cdots{}SR^{c_k}SR,\\[4pt]
&=&R^{a_1}(UR)^{a_2}\cdots{}R^{a_{2m-1}}(UR)^{a_{2m}}
\end{array}
$$
where $\frac{a}{c}=[a_1,\ldots,a_{2m}]=\llbracket{}c_1,\ldots,c_k\rrbracket$;

(ii) if $a<b$, then
$$
\begin{array}{rcl}
A&=&R^{c_1}S\,R^{c_{2}}S\cdots{}SR^{c_k},\\[4pt]
&=&R^{a_1}(UR)^{a_2}\cdots{}R^{a_{2m-1}}(UR)^{a_{2m}-1}R
\end{array}
$$
where $\frac{b}{d}=[a_1,\ldots,a_{2m}]=\llbracket{}c_1,\ldots,c_k\rrbracket$.
\end{cor}

Together with Proposition~\ref{MaTCompProp}, Theorem~\ref{SecondMainThm} implies an
explicit description of relations in the group $\PSL(2,\Z)$.
Every element $A\in\PSL(2,\Z)$ can be written in terms of the generators~$R$ and~$S$
(see formula~\eqref{NegMatGenEq}) as follows
$$
A=R^{c_1}S\,R^{c_{2}}S\cdots{}R^{c_n}S,
$$
where $c_i$ are some positive integers.
The following statement is actually equivalent to Theorem~\ref{SecondMainThm}.

\begin{cor}
One has 
$$R^{c_1}S\,R^{c_{2}}S\cdots{}R^{c_n}S=\Id$$ 
in $\PSL(2,\Z),$
if and only if $(c_1,\ldots,c_n)$ is the quiddity of 
a $3d$-dissection of an $n$-gon.
\end{cor}

Note that all of the above relations follow from the following two:
$$
S^2=\id,
\qquad
\left(RS
\right)^3=\id,
$$
since $\PSL(2,\Z)$ is known to be isomorphic to the free product of two cyclic groups
with the generators $S$ and~$RS$, namely
$\PSL(2,\Z)\simeq(\Z/2\Z)*(\Z/3\Z)$.

\begin{ex}
We go back to Example~\ref{CZKEx}.

(a) Consider first the matrix
$A'=\begin{pmatrix}
13&-9\\[2pt]
3&-2
\end{pmatrix}$.
It satisfies the condition from Part~(ii) of Theorem~\ref{SLDecThm}, and, indeed, we see that
$$
\frac{13}{3}=
[4,3]=\llbracket5,2,2\rrbracket.
$$
One then checks that $A'=M(5,2,2)$.

(b) The matrix 
$A=\begin{pmatrix}
10&3\\[2pt]
3&1
\end{pmatrix}$
is as in Corollary~\ref{PosDeCor}, Part~(i).
Since $\frac{10}{3}=
[3,3]=\llbracket4,2,2\rrbracket,$ one easily checks that $A=M(4,2,2,2,1,1)$, in accordance with~\eqref{CoeffciEq}.

(c) 
 The matrix 
$A=\begin{pmatrix}
3&10\\[2pt]
2&7
\end{pmatrix}$
is as in Corollary~\ref{PosDeCor}, Part~(ii).
Since $\frac{10}{7}=
[1,2,2,1]=\llbracket2,2,4\rrbracket,$ one checks that
$A=M(2,2,5,1,1)$,
and the conjugacy class of~$A$ is thus parametrized by $(6,2)$.
Note that this is the period of the negative continued fraction of
the attractive fixed point of~$A$:
$$
x_+=-1+\sqrt{6}=\llbracket2,\overline{2,6}\rrbracket.
$$.
\end{ex}

\subsection{Proof of Theorem~\ref{SLDecThm} and Proposition~\ref{MinCrProp}}\label{ProoDecSSec}

Theorem~\ref{SLDecThm} Part (i).
Consider the following two ``local surgery'' operations.

\begin{enumerate}
\item
Whenever the $n$-tuple $(c_1,\ldots,c_k)$ contains
a fragment $c_i,\,1\,,c_{i+2}$ with $c_i,c_{i+2}>1$,
the following ``Conway-Coxeter operation'' removes $1$ and
decreases the two neighboring entries by $1$:
\begin{equation}
\label{FirstSur}
(c_1,\ldots,c_i,\,1,\,c_{i+2},\ldots,c_k)
\mapsto
(c_1,\ldots,c_i-1,\,c_{i+2}-1,\ldots,c_k).
\end{equation}

\item
Whenever the $n$-tuple $(c_1,\ldots,c_k)$ contains
a fragment $c_i,\,1,\,1\,,c_{i+3}$
(with arbitrary~$c_i,c_{i+3}$),
the following operation reduces $k$ by $3$ producing the $(k-3)$-tuple
\begin{equation}
\label{SecSur}
(c_1,\ldots,c_i,\,1,\,\,1,\,c_{i+3},\ldots,c_k)
\mapsto
(c_1,\ldots,c_i+c_{i+3}-1,\ldots,c_k).
\end{equation}
\end{enumerate}

Note that these operations have already been used in Sections~\ref{QuidSSec} and~\ref{MoySec}.
It has already been checked,
that these operations preserve the element $M(c_1,\ldots,c_k)$ of $\PSL(2,\Z)$.

Given an arbitrary presentation $A=M(c_1,\ldots,c_k)$ with positive integers $c_i$,
applying the operations~(\ref{FirstSur}) and~(\ref{SecSur}) when this is possible
(and in arbitrary order),
the $k$-tuple $(c_1,\ldots,c_k)$ can be reduced to one of the case
$c_i\geq2$, except perhaps for $c_1$, or $c_1,c_2$ and $c_k$, or $c_{k-1},c_k$.

By multiplying  from the right by~$M(1)$ or~$M(1,\,1)$, we can furthermore normalize the sequence
$(c_1,\ldots,c_k)$ in such a way that it contains at most one entry $1$ at the end, i.e.,
that $c_{k-1}\geq2$.
This implies that the continued fraction $\llbracket{}c_1,\ldots,c_k\rrbracket$ is well defined.

Assume there are two different presentations of the same element,
$$
A=M(c_1,\ldots,c_k)=M(c'_1,\ldots,c'_{k}),
$$
with positive integers $c_i$ and minimal $k$,
so that the continued fraction 
$\llbracket{}c'_1,\ldots,c'_k\rrbracket=\llbracket{}c_1,\ldots,c_k\rrbracket$ is also well defined.
Then, without loss of generality, we can assume that
 $c_1>c'_1$.
 Since, for $k\geq3$, one has $\llbracket{}c_1,\ldots,c_k\rrbracket{}>c_1-1$
 and $\llbracket{}c'_1,\ldots,c'_{k}\rrbracket<c'_1$,
 this implies that 
 $
 \llbracket{}c_1,\ldots,c_k\rrbracket{}>\llbracket{}c'_1,\ldots,c'_{k}\rrbracket{},
 $
which contradicts the assumption.

Theorem~\ref{SLDecThm} Part (ii). The positivity of the coefficients of $A$ and the condition $\det A=1$ imply that
$\frac{a}{c}<\frac{b}{d}$. We give a geometric argument using results of Section \ref{TriangCFSec}.
The rationals  $\frac{a}{c}<\frac{b}{d}$ are linked by an edge in the Farey graph and one has the following two possible local pictures in $\mathbb{T}_{\frac{a}{c}}\cup  \mathbb{T}_{\frac{b}{d}}$. 

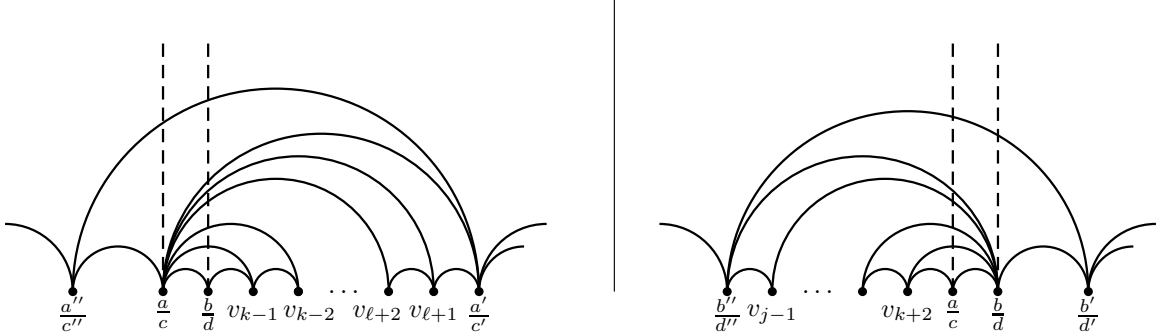
\psscalebox{1.0 1.0} 
{\psset{unit=0.75cm}
\begin{pspicture}(0,-2.915)(20.4,2.915)
\psdots[linecolor=black, dotsize=0.16](2.8,-2.285)
\psdots[linecolor=black, dotsize=0.16](3.6,-2.285)
\psdots[linecolor=black, dotsize=0.16](4.4,-2.285)
\psdots[linecolor=black, dotsize=0.16](5.2,-2.285)
\psdots[linecolor=black, dotsize=0.16](7.6,-2.285)
\psdots[linecolor=black, dotsize=0.16](8.4,-2.285)
\psdots[linecolor=black, dotsize=0.16](6.8,-2.285)
\psarc[linecolor=black, linewidth=0.04, dimen=outer](3.2,-2.285){0.4}{0.0}{180.0}
\psarc[linecolor=black, linewidth=0.04, dimen=outer](4.0,-2.285){0.4}{0.0}{180.0}
\psarc[linecolor=black, linewidth=0.04, dimen=outer](4.8,-2.285){0.4}{0.0}{180.0}
\psarc[linecolor=black, linewidth=0.04, dimen=outer](8.0,-2.285){0.4}{0.0}{180.0}
\psarc[linecolor=black, linewidth=0.04, dimen=outer](3.6,-2.285){0.8}{0.0}{180.0}
\psarc[linecolor=black, linewidth=0.04, dimen=outer](4.0,-2.285){1.2}{0.0}{180.0}
\psarc[linecolor=black, linewidth=0.04, dimen=outer](5.6,-2.285){2.8}{0.0}{180.0}
\psarc[linecolor=black, linewidth=0.04, dimen=outer](5.2,-2.285){2.4}{0.0}{180.0}
\rput(1.2,-2.685){$\frac{a''}{c''}$}
\rput(2.8,-2.685){$\frac{a}{c}$}
\rput(3.6,-2.685){$\frac{b}{d}$}
\rput(6.0,-2.285){$\ldots$}
\rput(7.6,-2.685){$v_{\ell+1}$}
\rput(6.6,-2.685){$v_{\ell+2}$}
\rput(4.4,-2.685){$v_{k-1}$}
\rput(5.4,-2.685){$v_{k-2}$}
\psarc[linecolor=black, linewidth=0.04, dimen=outer](9.6,-2.285){1.2}{90.0}{180.0}
\psline[linecolor=black, linewidth=0.04, linestyle=dashed, dash=0.17638889cm 0.10583334cm](2.8,-2.285)(2.8,2.115)(2.8,2.115)
\psline[linecolor=black, linewidth=0.04, linestyle=dashed, dash=0.17638889cm 0.10583334cm](3.6,-2.285)(3.6,2.115)(3.6,2.115)
\psarc[linecolor=black, linewidth=0.04, dimen=outer](11.6,-2.285){1.2}{0.0}{90.0}
\psdots[linecolor=black, dotsize=0.16](19.2,-2.285)
\psdots[linecolor=black, dotsize=0.16](16.0,-2.285)
\psdots[linecolor=black, dotsize=0.16](17.6,-2.285)
\psdots[linecolor=black, dotsize=0.16](16.8,-2.285)
\psdots[linecolor=black, dotsize=0.16](15.2,-2.285)
\psarc[linecolor=black, linewidth=0.04, dimen=outer](15.2,-2.285){2.4}{0.0}{180.0}
\psarc[linecolor=black, linewidth=0.04, dimen=outer](15.6,-2.285){2.0}{0.0}{180.0}
\psarc[linecolor=black, linewidth=0.04, dimen=outer](16.4,-2.285){1.2}{0.0}{180.0}
\psarc[linecolor=black, linewidth=0.04, dimen=outer](16.8,-2.285){0.8}{0.0}{180.0}
\psarc[linecolor=black, linewidth=0.04, dimen=outer](17.2,-2.285){0.4}{0.0}{180.0}
\rput(17.6,-2.685){$\frac{b}{d}$}
\rput(16.8,-2.685){$\frac{a}{c}$}
\rput(16.0,-2.685){$v_{k+2}$}
\psarc[linecolor=black, linewidth=0.04, dimen=outer](13.2,-2.285){0.4}{0.0}{180.0}
\psarc[linecolor=black, linewidth=0.04, dimen=outer](15.6,-2.285){0.4}{0.0}{180.0}
\psarc[linecolor=black, linewidth=0.04, dimen=outer](16.4,-2.285){0.4}{0.0}{180.0}
\psarc[linecolor=black, linewidth=0.04, dimen=outer](18.4,-2.285){0.8}{0.0}{180.0}
\psdots[linecolor=black, dotsize=0.16](12.8,-2.285)
\psdots[linecolor=black, dotsize=0.16](13.6,-2.285)
\rput(14.4,-2.285){$\ldots$}
\psline[linecolor=black, linewidth=0.04, linestyle=dashed, dash=0.17638889cm 0.10583334cm](16.8,-2.285)(16.8,2.115)(16.8,2.115)(16.8,2.115)
\psline[linecolor=black, linewidth=0.04, linestyle=dashed, dash=0.17638889cm 0.10583334cm](17.6,-2.285)(17.6,2.115)
\psarc[linecolor=black, linewidth=0.04, dimen=outer](16.0,-2.285){3.2}{0.0}{180.0}
\psarc[linecolor=black, linewidth=0.04, dimen=outer](2.0,-2.285){0.8}{0.0}{180.0}
\psdots[linecolor=black, dotsize=0.16](1.2,-2.285)
\psdots[linecolor=black, dotsize=0.16](5.2,-2.285)
\psarc[linecolor=black, linewidth=0.04, dimen=outer](0.0,-2.285){1.2}{0.0}{90.0}
\psarc[linecolor=black, linewidth=0.04, dimen=outer](20.4,-2.285){1.2}{90.0}{180.0}
\psline[linecolor=black, linewidth=0.02](10.8,-2.285)(10.8,2.915)
\rput(13.6,-2.685){$v_{j-1}$}
\psarc[linecolor=black, linewidth=0.04, dimen=outer](9.2,-2.285){0.8}{90.0}{180.0}
\psarc[linecolor=black, linewidth=0.04, dimen=outer](20.0,-2.285){0.8}{90.0}{180.0}
\psarc[linecolor=black, linewidth=0.04, dimen=outer](4.8,-2.285){2.0}{0.0}{180.0}
\psarc[linecolor=black, linewidth=0.04, dimen=outer](7.2,-2.285){0.4}{0.0}{180.0}
\psarc[linecolor=black, linewidth=0.04, dimen=outer](4.8,-2.285){3.6}{0.0}{180.0}
\rput(8.4,-2.685){$\frac{a'}{c'}$}
\rput(19.2,-2.685){$\frac{b'}{d'}$}
\rput(12.8,-2.685){$\frac{b''}{d''}$}
\end{pspicture}
}

The cases split as follows: in the left case $a<b$ whereas in the right case $a>b$. 
In the right case~$\frac{b}{d}$ is the previous convergent in the expansion of $\frac{a}{c}$ as negative continued fraction. In other words, if 
$\frac{a}{c}=\llbracket{}c_1,\ldots,c_k\rrbracket$ then $\frac{b}{d}=\llbracket{}c_1,\ldots,c_{k-1}\rrbracket$; which gives $A=M(c_{1}, \ldots, c_{k})$ according to 
Proposition~\ref{MatConv}.

Theorem~\ref{SLDecThm} is proved.

To prove Proposition~\ref{MinCrProp}, first note that
the operations~\eqref{FirstSur} and \eqref{SecSur} allow one reduce
any presentation $A=M(c_1,\ldots,c_k)$ to the form with $c_i\geq2$, except perhaps for
the ends of the sequence.
Hence the ``only if'' part.
The proof of the ``if'' part is similar to that of Theorem~\ref{SLDecThm}, Part~(i).

\subsection{Minimal presentation and Farey $n$-gon}\label{JustSec}
Given a matrix $A\in\PSL(2,\Z)$, we explain how to recover the coefficients $(c_{1}, \ldots, c_{k})$
in a combinatorial manner. More precisely, the coefficients can be interpreted as the quiddity of some triangulated polygon.

We focus on the case of matrices of the form
$$
A=\begin{pmatrix}
a&-b\\[4pt]
c&-d
\end{pmatrix}
$$
with $a,b,c,d>0$.
Such a matrix $A$ defines two (positive) rationals $\frac{a}{c}<\frac{b}{d}$ that are linked by an edge 
in the Farey tessellation.
Similarly to what was done in Section~\ref{TrsFarSec}, 
one draws two vertical lines from~$\frac{a}{c}$ and~$\frac{b}{d}$ in the Farey tessellation 
and collects all the triangles crossed by these lines.
One thus obtains a triangulated Farey $n$-gon that we denote $\mathbb{T}_{A}$.

Note that $\mathbb{T}_{A}$ is the union $\mathbb{T}_{\frac{a}{c}}\cup\mathbb{T}_{\frac{b}{d}}$ 
of the triangulations defined in Section \ref{CombIntCFSec}.

We label the vertices of $\mathbb{T}_{A}$ in decreasing order so that 
$$
v_{1}= \frac10,\ldots, v_{k}=\frac{b}{d},\, v_{k+1}=\frac{a}{c}, \ldots, v_{n}= \frac01,
$$
and we denote by $(c_{1},\ldots, c_{k})$ the quiddity sequence attached to the first~$k$ vertices. 
One has
$$
A=M(c_{1},\ldots, c_{k}).
$$

\begin{ex}
\label{CFMatEx}
For the matrix
$
A=\begin{pmatrix}
2&-5\\
1&-2
\end{pmatrix}
$, we obtain the triangulation of Figure~2.

\begin{figure}
\label{triangAF}
\begin{center}
\psscalebox{1.0 1.0} 
{\psset{unit=0.7cm}
\begin{pspicture}(0,-3.9025)(13.797115,3.9025)
\psdots[linecolor=black, dotsize=0.08](3.2985578,-3.2975)
\psdots[linecolor=black, dotsize=0.08](2.4985578,-3.2975)
\psdots[linecolor=black, dotsize=0.08](1.6985577,-3.2975)
\psdots[linecolor=black, dotsize=0.2](0.098557696,-3.2975)
\psdots[linecolor=black, dotsize=0.08](4.0985575,-3.2975)
\psdots[linecolor=black, dotsize=0.08](4.8985577,-3.2975)
\psdots[linecolor=black, dotsize=0.08](5.698558,-3.2975)
\psdots[linecolor=black, dotsize=0.08](6.4985576,-3.2975)
\psdots[linecolor=black, dotsize=0.08](8.098557,-3.2975)
\psdots[linecolor=black, dotsize=0.08](9.698558,-3.2975)
\psdots[linecolor=black, dotsize=0.08](11.298557,-3.2975)
\psdots[linecolor=black, dotsize=0.08](12.098557,-3.2975)
\psdots[linecolor=black, dotsize=0.08](12.898558,-3.2975)
\psarc[linecolor=black, linewidth=0.08, dimen=outer](6.8985577,-3.2975){6.8}{0.0}{180.0}
\psarc[linecolor=black, linewidth=0.02, dimen=outer](6.8985577,-3.2975){0.4}{0.0}{180.0}
\psarc[linecolor=black, linewidth=0.02, dimen=outer](7.698558,-3.2975){0.4}{0.0}{180.0}
\psarc[linecolor=black, linewidth=0.02, dimen=outer](8.498558,-3.2975){0.4}{0.0}{180.0}
\psarc[linecolor=black, linewidth=0.02, dimen=outer](9.298557,-3.2975){0.4}{0.0}{180.0}
\psarc[linecolor=black, linewidth=0.02, dimen=outer](10.098557,-3.2975){0.4}{0.0}{180.0}
\psarc[linecolor=black, linewidth=0.02, dimen=outer](10.898558,-3.2975){0.4}{0.0}{180.0}
\psarc[linecolor=black, linewidth=0.02, dimen=outer](11.698558,-3.2975){0.4}{0.0}{180.0}
\psarc[linecolor=black, linewidth=0.02, dimen=outer](12.498558,-3.2975){0.4}{0.0}{180.0}
\psarc[linecolor=black, linewidth=0.02, dimen=outer](13.298557,-3.2975){0.4}{0.0}{180.0}
\psarc[linecolor=black, linewidth=0.08, dimen=outer](0.4985577,-3.2975){0.4}{0.0}{180.0}
\psarc[linecolor=black, linewidth=0.02, dimen=outer](1.2985576,-3.2975){0.4}{0.0}{180.0}
\psarc[linecolor=black, linewidth=0.02, dimen=outer](2.0985577,-3.2975){0.4}{0.0}{180.0}
\psarc[linecolor=black, linewidth=0.02, dimen=outer](2.8985577,-3.2975){0.4}{0.0}{180.0}
\psarc[linecolor=black, linewidth=0.02, dimen=outer](3.6985576,-3.2975){0.4}{0.0}{180.0}
\psarc[linecolor=black, linewidth=0.02, dimen=outer](4.4985576,-3.2975){0.4}{0.0}{180.0}
\psarc[linecolor=black, linewidth=0.02, dimen=outer](5.2985578,-3.2975){0.4}{0.0}{180.0}
\psarc[linecolor=black, linewidth=0.02, dimen=outer](6.0985575,-3.2975){0.4}{0.0}{180.0}
\psarc[linecolor=black, linewidth=0.02, dimen=outer](1.6985577,-3.2975){0.8}{0.0}{180.0}
\psarc[linecolor=black, linewidth=0.02, dimen=outer](3.2985578,-3.2975){0.8}{0.0}{180.0}
\psarc[linecolor=black, linewidth=0.02, dimen=outer](4.8985577,-3.2975){0.8}{0.0}{180.0}
\psarc[linecolor=black, linewidth=0.02, dimen=outer](6.4985576,-3.2975){0.8}{0.0}{180.0}
\psarc[linecolor=black, linewidth=0.08, dimen=outer](8.098557,-3.2975){0.8}{0.0}{180.0}
\psarc[linecolor=black, linewidth=0.08, dimen=outer](9.698558,-3.2975){0.8}{0.0}{180.0}
\psarc[linecolor=black, linewidth=0.02, dimen=outer](11.298557,-3.2975){0.8}{0.0}{180.0}
\psarc[linecolor=black, linewidth=0.02, dimen=outer](12.898558,-3.2975){0.8}{0.0}{180.0}
\psarc[linecolor=black, linewidth=0.02, dimen=outer](2.4985578,-3.2975){1.6}{0.0}{180.0}
\psarc[linecolor=black, linewidth=0.02, dimen=outer](4.8985577,-3.2975){0.8}{0.0}{180.0}
\psarc[linecolor=black, linewidth=0.02, dimen=outer](5.698558,-3.2975){1.6}{0.0}{180.0}
\psarc[linecolor=black, linewidth=0.06, dimen=outer](8.898558,-3.2975){1.6}{0.0}{180.0}
\psarc[linecolor=black, linewidth=0.08, dimen=outer](12.098557,-3.2975){1.6}{0.0}{180.0}
\psarc[linecolor=black, linewidth=0.08, dimen=outer](4.0985575,-3.2975){3.2}{0.0}{180.0}
\psarc[linecolor=black, linewidth=0.06, dimen=outer](10.498558,-3.2975){3.2}{0.0}{180.0}
\psarc[linecolor=black, linewidth=0.06, dimen=outer](7.2985578,-3.2975){6.4}{0.0}{180.0}
\rput(0.098557696,-3.6975){$\frac01$}
\rput(0.89855766,-3.6975){$\frac11$}
\rput(1.6985577,-3.6975){$\frac54$}
\rput(2.4985578,-3.6975){$\frac43$}
\rput(3.2985578,-3.6975){$\frac75$}
\rput(4.0985575,-3.6975){$\frac32$}
\rput(4.8985577,-3.6975){$\frac85$}
\rput(5.698558,-3.6975){$\frac53$}
\rput(6.4985576,-3.6975){$\frac74$}
\rput(7.2985578,-3.6975){$\frac21$}
\rput(8.098557,-3.6975){$\frac73$}
\rput(8.898558,-3.6975){$\frac52$}
\rput(9.698558,-3.6975){$\frac83$}
\rput(10.498558,-3.6975){$\frac31$}
\rput(11.298557,-3.6975){$\frac72$}
\rput(12.098557,-3.6975){$\frac41$}
\rput(12.898558,-3.6975){$\frac51$}
\rput(13.698558,-3.6975){$\frac10$}
\psline[linecolor=black, linewidth=0.04, linestyle=dashed, dash=0.17638889cm 0.10583334cm](7.2985578,-3.2975)(7.2985578,3.9025)(7.2985578,3.9025)
\psline[linecolor=black, linewidth=0.04, linestyle=dashed, dash=0.17638889cm 0.10583334cm](8.898558,-3.2975)(8.898558,3.9025)
\psdots[linecolor=black, dotsize=0.2](0.89855766,-3.2975)
\psdots[linecolor=black, dotsize=0.2](7.2985578,-3.2975)
\psdots[linecolor=black, dotsize=0.2](8.898558,-3.2975)
\psdots[linecolor=black, dotsize=0.2](10.498558,-3.2975)
\psdots[linecolor=black, dotsize=0.2](13.698558,-3.2975)
\end{pspicture}}
\end{center}
\caption{ The triangulation
$\mathbb{T}_{A}$.}
\end{figure}
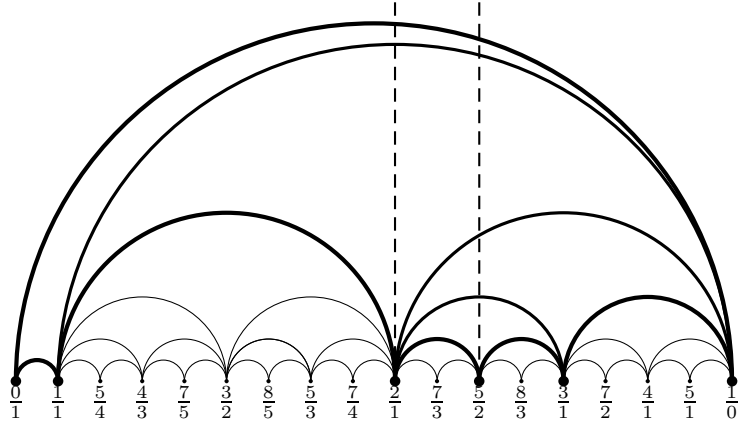

The corresponding Farey hexagon is 
$
\left(
\frac{1}{0},\,
\frac{3}{1},\,
\frac{5}{2},\,
\frac{2}{1},\,
\frac{1}{1},\,
\frac{0}{1}
\right)$
and the triangulation  $\mathbb{T}_{A}$ can be pictured as follows:
$$
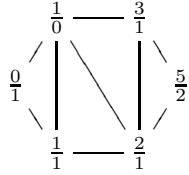
\xymatrix @!0 @R=0.45cm @C=0.55cm
{
&\frac{1}{0}\ar@{-}[ldd]\ar@{-}[rr]\ar@{-}[rrdddd]\ar@{-}[dddd]
&& \frac{3}{1}\ar@{-}[rdd]\ar@{-}[dddd]&\\
\\
\frac{0}{1}\ar@{-}[rdd]
&&&&  \frac{5}{2}\ar@{-}[ldd]\\
\\
& \frac{1}{1}\ar@{-}[rr]&& \frac{2}{1}
}
$$
The quiddity sequence at the vertices $\frac{1}{0},
\frac{3}{1},
\frac{5}{2}$ is $(3,2,1)$ so that we deduce
$$
A=\begin{pmatrix}
2&-5\\
1&-2
\end{pmatrix}
=M(3,2,1).$$
\end{ex}

\subsection{Further examples: Cohn matrices}

Let us give more examples of interesting matrices.

\begin{ex}
(a)
Recall that an element~$A\in\PSL(2,\Z)$ is called parabolic if
$\tr\!(A)=2$; a parabolic element is conjugate to~$R^a$ with $a\in\Z$.
The parabolic element $R^a$, for~$a\geq0$ has the following minimal presentation
$
R^a=M(a+1,1,1),
$
while the minimal presentations of $L^a$ and $R^{-a}$ 
are as follows:
$$
L^a=
\begin{pmatrix}
1&0\\[2pt]
a&1
\end{pmatrix}=
M(1,\underbrace{2,\ldots,2}_{a},1,1),
\qquad
R^{-a}=
\begin{pmatrix}
1&-a\\[2pt]
0&1
\end{pmatrix}=M(1,1,\underbrace{2,\ldots,2}_{a},1).
$$
Note that the above equality hold in~$PSL(2,\Z)$, i.e., the matrix equalities
are up to a sign.
The elements $L^a$ and $R^{-a}$ belong to the same conjugacy class
parametrized by~$(2,2,\ldots,2)$.
The element~$R^a$ belongs to a different conjugacy class.

(b)
The minimal presentation of the continued fraction matrix
$M^+(a_1,\ldots,a_{2m})$ is given by~(\ref{NegRegMat}).

(c)
The famous {\it Cohn matrices} are the triples of matrices, 
$(A,AB,B)$ in which the triples of Markov numbers appear
both, as right upper entry, and as $\frac{1}{3}$ of the traces.
It is known (see~\cite{Aig}) that such matrices are enumerated by
$(n,t)$, where $n\in\Z$ and $t$ is a rational $0\leq t\leq1$.
The initial triple of Cohn matrices given by
$$
A(n)=\begin{pmatrix}
n&1\\[4pt]
3n-n^2-1&3-n
\end{pmatrix}
$$ 
and $B(n):=A(n)A(n+1)$
corresponds to the Markov triple 
$(1,5,2)=\left(\frac{1}{3}\tr(A),\frac{1}{3}\tr(AB),\frac{1}{3}\tr(B)\right)$.
Other triples of Cohn matrices are given by the products of the matrices from the initial
triple, encoded by a tree
isomorphic to the Farey (or Stern-Brocot) tree of rationals in $[0,1]$, starting from
the triple $(0,\frac{1}{2},1)$.

The minimal presentation of the initial matrices with $n\geq2$ is
$$
\begin{array}{rcl}
A(n)&=&
M(1,1,n-1,\underbrace{2,\ldots,2}_{n},1,1),\\[4pt]
B(n)&=&
M(1,1,n-1,3,\underbrace{2,\ldots,2}_{n},1,1),\\[4pt]
A(n)B(n)&=&
M(1,1,n-1,2,4,\underbrace{2,\ldots,2}_{n},1,1).
\end{array}
$$
Furthermore,
$$
\begin{array}{rcl}
A(n)^2B(n)&=&
M(1,1,n-1,2,3,4,\underbrace{2,\ldots,2}_{n},1,1),\\[4pt]
A(n)B(n)^2&=&
M(1,1,n-1,2,4,2,4,\underbrace{2,\ldots,2}_{n},1,1),
\end{array}
$$
etc.

We see that all of these matrices with different~$n$ are conjugate to each other,
the conjugacy classes of~$A$ and~$B$ being parametrized by~$(3)$
and~$(4,2)$, respectively.
\end{ex}

\subsection{The $3d$-dissection of a matrix}
We apply Theorem~\ref{SLDecThm} in order to associate
a $3d$-dissection to every element $A\in\SL(2,\Z)$.
Our construction is as follows.

Writing $A$ and $A^{-1}$ in the canonical minimal form
$A=M(c_1,\ldots,c_k)$ and
$A^{-1}=M(c'_1,\ldots,c'_\ell),$
one obtains a~$(k+\ell)$-tuple of positive integers
$(c_1,\ldots,c_k,\,c'_1,\ldots,c'_\ell)$.
Since
$$
M(c_1,\ldots,c_k,\,c'_1,\ldots,c'_\ell)=
M(c_1,\ldots,c_k)\,M(c'_1,\ldots,c'_\ell)=\Id,
$$
Theorem~\ref{SecondMainThm} implies that
this is a quiddity of some
$3d$-dissection.

\begin{ex}
(a)
The matrix~$S=M(1,1,2,1,1)$ corresponds to
the quiddity of the hexagonal dissection of a decagon:
$$
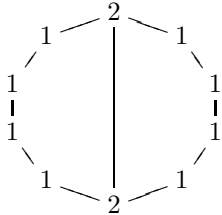
\xymatrix @!0 @R=0.32cm @C=0.45cm
 {
&&&{2}\ar@{-}[dddddddd]\ar@{-}[lld]\ar@{-}[rrd]&
\\
&{1}\ar@{-}[ldd]&&&& {1}\ar@{-}[rdd]\\
\\
{1}\ar@{-}[dd]&&&&&&{1}\ar@{-}[dd]\\
\\
1\ar@{-}[rdd]&&&&&&1\ar@{-}[ldd]\\
\\
&1\ar@{-}[rrd]&&&& 1\ar@{-}[lld]\\
&&&2&
}
$$

(b)
For the matrix $R$
one has $R=M(2,1,1)$ (up to a sign) and $R^{-1}=M(1,1,2,1)$.
This leads to the dissection of a heptagon:
$$
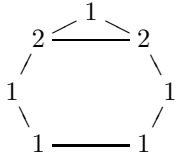
\xymatrix @!0 @R=0.35cm @C=0.35cm
{
&&&1\ar@{-}[rrd]\ar@{-}[lld]
\\
&2\ar@{-}[ldd]\ar@{-}[rrrr]&&&& 2\ar@{-}[rdd]&\\
\\
1\ar@{-}[rdd]&&&&&& 1\ar@{-}[ldd]\\
\\
&1\ar@{-}[rrrr]&&&&1
}
$$

(c) 
Consider the following elements
$$
A=\begin{pmatrix}
2&1\\[2pt]
1&1
\end{pmatrix},
\qquad
B=\begin{pmatrix}
5&2\\[2pt]
2&1
\end{pmatrix}
$$ 
which are the simplest Cohn matrices (with $n=2$).
One has the following presentations:
$$
A=M(2,2,1,1),
\quad
B=M(3,2,2,1,1),
\qquad
A^{-1}=M(1,1,3,1),
\quad
B^{-1}=M(1,1,4,2,1).
$$
The corresponding quiddities are those of the dissected octagon and decagon:
$$
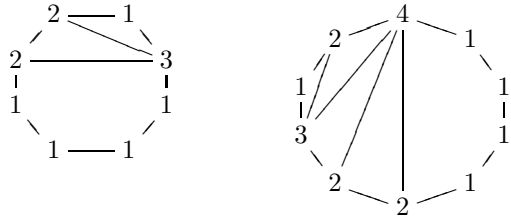
\xymatrix @!0 @R=0.30cm @C=0.5cm
 {
&2\ar@{-}[ldd]\ar@{-}[rrrdd]\ar@{-}[rr]&& 1\ar@{-}[rdd]\\
\\
2\ar@{-}[dd]\ar@{-}[rrrr]&&&&3\ar@{-}[dd]\\
\\
1\ar@{-}[rdd]&&&&1\ar@{-}[ldd]\\
\\
&1\ar@{-}[rr]&& 1
}
\qquad
\qquad
 \xymatrix @!0 @R=0.32cm @C=0.45cm
 {
&&&4\ar@{-}[lld]\ar@{-}[dddddddd]\ar@{-}[rrd]&
\\
&2\ar@{-}[ldd]\ar@{-}[ldddd]&&&& 1\ar@{-}[rdd]\\
\\
1\ar@{-}[dd]&&&&&&1\ar@{-}[dd]\\
\\
3\ar@{-}[rrruuuuu]\ar@{-}[rdd]&&&&&&1\ar@{-}[ldd]\\
\\
&2\ar@{-}[rrd]\ar@{-}[rruuuuuuu]&&&& 1\ar@{-}[lld]\\
&&&2&
}
$$
\end{ex}

\bigbreak \noindent
{\bf Acknowledgements}.
We are grateful to Charles Conley, Vladimir Fock, Sergei Fomin,
Alexey Klimenko, and Sergei Tabachnikov
for multiple stimulating and enlightening discussions.
We are grateful to the referee for a number of helpful remarks and suggestions.
This paper was partially supported by the ANR project SC3A, ANR-15-CE40-0004-01.

\end{document}